\documentclass[12pt]{article}

\usepackage{latexsym,amssymb}

\def\bea{\begin{eqnarray}}
\def\eea{\end{eqnarray}}
\def\beann{\begin{eqnarray*}}
\def\eeann{\end{eqnarray*}}
\def\beeq#1{\begin{equation}{#1}\end{equation}}

\def\tr{^{\rm T}}
\def\Real{{I\!\!R}}
\def\Compl {{C\!\!\!\!I}}

\def\ca{{{\cal A}}}

\def\dimvarsigma{d}
\def\dimz{m} 
\def\dimx{n} 
\def\dimy{r} 
\def\dimmu{s} 
\def\dimeta{\nu} 
\def\dimxi{p}
\def\dimzeta{o}

\def\be{\begin{equation}}
\def\ee{\end{equation}}
\def\ba{\begin{array}}
\def\ea{\end{array}}
\def\bea{\begin{eqnarray}}
\def\eea{\end{eqnarray}}
\def\beann{\begin{eqnarray*}}
\def\eeann{\end{eqnarray*}}

\newtheorem{theorem}{Theorem}

\newtheorem{proposition}{Proposition}
\newtheorem{lemma}{Lemma}
\newtheorem{definition}{Definition}

\newtheorem{remark_temp}[theorem]{Remark}
\newenvironment{remark}{\begin{remark_temp} \upshape}{ \end{remark_temp}}

\newenvironment{proof}{{\it Proof.}}{\hfill $\triangleleft$ \smallskip}

\def\real{\mbox{\rm I\kern-0.2em R}}
\def\dst{\displaystyle}
\title{Robust Asymptotic Stabilization of Nonlinear Systems\\ with Non-Hyperbolic
Zero Dynamics}
\author {L. Marconi$^{\circ}$, L. Praly$^\ast$, A. Isidori$^{\circ\dag}$\\[3mm]
$^{\circ}$ C.A.SY. -- DEIS,
University of Bologna, Bologna, Italy.\\[1mm]
 $^{\ast}$ \'Ecole des Mines de Paris, Fontainebleau, France.\\[1mm]
  $^{\dag}$DIS,
 Universit\`{a} di Roma ``La Sapienza'',  Rome, Italy.
}
 
\begin{document}

 \maketitle

\begin{abstract}
 In this paper we present a general tool to handle the presence of
 zero dynamics which are asymptotically but not locally exponentially stable in problems of
 robust nonlinear stabilization by output feedback. We show how it is possible
 to design {\em locally Lipschitz} stabilizers under conditions
 which only rely upon a partial detectability assumption on the controlled plant,
 by obtaining a robust stabilizing paradigm which is not based on
 design of observers and separation principles.
 The main design idea comes from recent achievements in the field
 of output regulation  and specifically in the design of nonlinear internal
 models.
\end{abstract}

\section{Introduction}
The problem of output feedback stabilization in the large for nonlinear systems has been the
subject of a remarkable research attempt in the last twenty years
 or so (see \cite{IsBook2}). The attempt has been initially turned to
 identify systematic design procedures for {\em state-feedback}
 stabilization of specific classes of nonlinear systems. To this respect
 it is worth mentioning the research current focused on back-stepping design
 procedures for lower triangular nonlinear systems with \cite{KoKo}
 for the global case and \cite{TP} for the semiglobal case.
 Then, the attention of the researchers shifted to the
 identification of partial-state and {\em output feedback} stabilization algorithms
 mainly addressed in a semi-global sense due to intrinsic limitations
 characterizing this class of problems (see \cite{Daya}).
 Within the number of research directions undertaken in this field, a
 special role has been played by {\em nonlinear separation
 principles} based on the design
 of an explicit full state observer (see \cite{TPSCL}). The main
 limitation of this approach is, thought, the lack of a
 guaranteed level of robustness of the resulting controller mainly due
 to the absence of a well-established theory of robust nonlinear state
 observers. Furthermore, full state observability of the controlled plant
 is not, in principle, a necessary condition for output feedback stabilization.
 A step forward to overcome these limitations has been taken in \cite{TP}
 with the definition of Uniform Completely Observable (UCO) state-feedback control
 law, namely a stabilizing state dependent law which can be
 expressed as nonlinear function of the control input and output and
 their time derivatives. In this case the issue is not to estimate the full-state but
 rather to reproduce directly the stabilizing law through the estimation of the
 input-output derivatives. This, in
 \cite{TP}, has been achieved by a mix of back-stepping and partial-state
 observation techniques yielding an output feedback stabilizer which is
 {\em robust} in the measure in which the UCO function does not depend on the
 uncertainties and the UCO control law is {\em vanishing} on the desired asymptotic
 attractor. Furthermore the {\em asymptotic} features of the resulting closed-loop
 system are subjected to the requirement that the initial state-feedback UCO-based
 closed-loop system is locally exponential stable. Practical
 stability must be accepted otherwise (see also \cite{Celani} at this regard).
   The latter limitation may be
 overtaken with the design of a local nonlinear observer in the spirit of
 \cite{TPSCL} by resuming again nonlinear separation principles. However,
 so-doing, the same limitations outlined before come out.\\
  Exponential stability assumptions are recurrent in several contexts of nonlinear control
  literature while studying asymptotic behaviors of nonlinear systems. Backstepping (\cite{Freeman},
  \cite{TP}), in which the backstepped control law is usually required to exponentially stabilize the
  controlled dynamics, singular perturbation (\cite{OReilly}, \cite{TeMoNe03}), in which the so-called boundary layer
  system is required to posses an exponentially stable attractor,  averaging (\cite{Khalil}, \cite{TeMoNe03}),
  in which exponential stability of the so-called averaged system is needed,  stabilization by output feedback
  (\cite{IsBook2}, \cite{TP}), in which  hyperbolic minimum-phase assumptions are usually required, are just a few contexts,
  involving problems of both synthesis and analysis
  of nonlinear systems, where the possibility of concluding asymptotic (and not only practical) results relies
  upon requirements that certain dynamics fulfill exponential stability  assumptions.
  A particular mention must be done for the design of output feedback stabilizers for nonlinear systems which
  can be written in so-called normal form (see \cite{IsBook2}). In this context exponential stability
  of the so-called zero dynamics (that is, hyperbolic zero dynamics) is very often a crucial pre-requisite if one
  is willing to address robust output feedback stabilization by means of {\em locally Lipschitz}
  regulators. \\
   In this paper we present a tool to handle the presence of not necessarily hyperbolic
   zero dynamics in the stabilization of nonlinear systems by output feedback. As particular application,
   the tool is then used to extend the main "UCO" results presented in \cite{TP} by, so doing,
   overtaking the obstacle of exponential stability in the backstepping procedure and output derivatives observer
   design.  More specifically, by means of the mathematical
 tools which have been developed in a context of nonlinear output regulation (see \cite{SICON}), \cite{BI03}),
 we show how the
 design of a dynamic output feedback control law which
 {\em asymptotically} stabilizes a compact attractor can be obtained by
 starting from a UCO state-feedback control law which does not
 necessarily stabilize in exponential way  the desired asymptotic attractor and
 which is not necessarily vanishing on it. We will show that these
 limitations can be removed by means of design techniques   aiming to
 robustly get rid of interconnections terms between nonlinear dynamics
 arising in the stability analysis which are not vanishing on the desired
 asymptotic attractor and which, as a consequence, can not be dominated only by means
 of high-gain.
 This will lead to identify a dynamic back-stepping and an
 extended partial-state observer algorithms which
 embed solution techniques typical of internal model-based design.\\
 This work is organized as follows. In the next section the framework and the general result is
 given. Then, Section 3 discusses the proposed framework and solution by properly framing the
 result in the existing literature. Section 4, articulated in three subsections, is focused on the application of the
 proposed tool in the UCO context presented in \cite{TP}. Then, Section 5 presents a few
 conditions, obtained by mild adaptation of results proposed in the output regulation literature,
 useful to construct the dynamic regulator which solve the problem discussed in Section 2. Finally
 Section 6 and 7 concludes with an example and final remarks.\\[1mm]
{\bf Notation}  For $x \in \Real^n$, $|x|$ denotes the Euclidean
norm and, for $\cal C$ a closed
 subset of $\Real^n$, $|x|_{\cal C}=\min_{y \in \cal C}|x-y|$ denotes the
 distance of $x$ from $\cal C$. For $\cal S$ a
subset of $\Real^n$, $\mbox{cl}\cal S$ and $\mbox{int}\cal S$ are
 the closure of $\cal S$ and the interior of $\cal S$
 respectively,
 and $\partial {\cal S}$ its boundary.
 A class-$\cal KL$ function $\beta(\cdot,\cdot)$
satisfying $|s| \leq d$ $\Rightarrow$ $\beta(t,s) \leq N
e^{-\lambda t} |s|$ for some positive $d$, $N$, $\lambda$ is said
to be a locally exponential class-$\cal KL$ function. For a
locally Lipschitz system of the form $\dot z = f(z)$  the value at
time $t$ of the solution passing through $z_0$ at time $t=0$ will
be written as $\phi_f(t,z_0)$ or, if the initial condition and the
system are clear from the context, as $z(t)$ or $z(t,z_0)$. For a
smooth system $\dot x = f(x)$, $x \in \Real^\dimx$, a compact set
$\ca$ is said to be LAS($\cal X$) (respectively LES($\cal X$)),
with ${\cal X} \subset \Real^\dimx$ a compact set, if it is
locally asymptotically (respectively exponentially) stable with a
domain of attraction containing ${\cal X}$. By ${\cal D}(\ca)$ we
denote the domain of attraction of $\ca$ if the latter is LAS/LES
for a given dynamics. For a function $f:\Real^n \to \Real^n$ and a
differentiable real-valued function $q:\Real^n \to \Real$,
$L_fq(x)$ denote the Lie derivative at $x$ of $q$ along $f$. For a
smooth system $\dot x = f(x)$, $x \in \Real^\dimx$ the
$\omega$-{\em limit set} of a subset $B \subset \Real^\dimx$,
written $\omega(B)$, is the set of all points $x\in\Real^\dimx$
for which there exists a sequence of pairs $(x_k,t_k)$, with
$x_k\in B$ and $t_k\to \infty$ as $k\to \infty$, such that
$\lim_{k\to \infty}\phi_f(t_k,x_k)=x$.

\section{The framework and the main result} \label{SecFramework}
 The main goal of this paper is to present a design tool to
 handle the presence of asymptotically but not necessarily
 exponentially stable zero dynamics in robust output-feedback stabilization
 problems of nonlinear systems.
 Although the tool we are going to present lends itself to be useful in a
 significant variety of control scenarios, in order to keep confined the discussion
 while maintaining a certain degree of generality, we focus our attention on the class of
 {\em smooth} systems of the form
 \beeq{\label{plantstab}
  \ba{rcll}
  \dot x &=& f(w, x, y) & \quad  x \in \Real^\dimx\,, \; n \geq 0 \\
  \dot y &=&  \kappa \, A y  + B(q(w, x,y) + v) & \quad y \in
  \Real^\dimy\,, \; r \geq 1
  \ea
 }
 with measurable output
 \[
   y_m = C y \qquad y_m\in \Real
 \]
 in which the linear system $(A,B,C)$ is assumed to have relative degree $r$ with the pair $(A,C)$
 observable,
 $\kappa$ is a positive design parameter and $v$ is a control input.
 In the previous system the variable $w \in \Real^\dimmu$
 represents an exogenous variable which is governed by
 \beeq{ \label{exosys}
  \dot w = s(w) \qquad w \in W \subset \Real^\dimmu
 }
 with $W$ a compact set which is {\em invariant} for
 (\ref{exosys}).   As a particular case, the signals $w(t)$ generated by
 (\ref{exosys}) may be constant signals, i.e. $s(w)\equiv 0$, namely constant uncertain parameters taking
 value in the set $W$ and  affecting the system (\ref{plantstab}).
 In general, the variables $w$ can be considered as exogenous signals which, depending on the
 considered control scenario, may represent references to be
 tracked and/or disturbances to be rejected.

 \begin{remark}
 As a consequence of the fact that $W$ is a (forward and backward)
 invariant set for (\ref{exosys}), the {\em closed} cylinder ${\cal C}_{\dimx + \dimy}:=
 W \times \Real^{\dimx + \dimy}$ is invariant for (\ref{plantstab}),(\ref{exosys}).
 Thus it is natural to regard
 system (\ref{plantstab}), (\ref{exosys}) on ${\cal C}_{\dimx + \dimy}$ and endow the latter with the relative topology.
 This will be done from now on by referring to system (\ref{plantstab}),(\ref{exosys}).
 Analogously, the dynamics described by the first $\dimx$
 equations of (\ref{plantstab}) and by (\ref{exosys}) will be thought as evolving on the closed set
 ${\cal C}_{\dimx}:=
 W \times \Real^{\dimx}$ which will be endowed with the relative
 topology. $\triangleleft$
 \end{remark}

 We shall study the previous system under the following ``minimum-phase"
 assumption.

 {\em Assumption} There exist compact sets $\ca \subset
 {\cal C}_\dimx$ which is locally asymptotically stable for the
 system
 \beeq{\label{zd1}
 \ba{rcl}
 \dot w &=& s(w)\\
 \dot x &=& f(w, x,0)\,. \quad \triangleleft
 \ea
 }
 Under this assumption, there exists a compact set ${\cal X} \subset {\cal
 C}_n$ such that $\ca \subset \mbox{int} \cal X$ and $\ca$ is
 LAS($\cal X$) for system (\ref{zd1}).\\
 In this framework we consider the output feedback stabilization problem
 which consists of designing a {\em locally Lipschitz} regulator of the form
  \beeq{\label{etasys}
  \dot \eta = \varphi_k(\eta, y_{\rm m}) \qquad v=\rho_k(\eta,y_{\rm m})
  \qquad \eta \in \Real^\dimeta\,,
  }
  and, given arbitrary bounded sets
  ${\cal Y} \subset \Real^\dimy$ and ${\cal N} \subset \Real^\dimeta$,
  a positive $\kappa^\star$, such that for all $ \kappa \geq \kappa^\star$ and
  for some ${\cal B} \subset \Real^{\dimeta + \dimx}$ the set ${\cal B} \times \{0\}$
  is LAS(${\cal N} \times {\cal X} \times {\cal Y}$) for the closed-loop system
  (\ref{plantstab}), (\ref{etasys}).

 The important point here is that $\varphi_k$ and $\rho _k$ must be
locally Lipschitz. This restriction has strong practical
motivations like sensitivity to noise or numeric and discrete time
implementation.

  The goal of the following part is to present a result regarding the solution of
  the robust stabilization problem formulated above. In order to ease the notation,
  in the following we shall drop in (\ref{plantstab}) the dependence from the variable $w$ which, in turn, will be
  thought as embedded in the variable $x$ (with the latter varying in the set ${\cal C}_{\dimx}$). This, with a mild abuse of notation, will allow us  to rewrite
  system (\ref{plantstab}) and (\ref{exosys}) in the more compact form
  \beeq{\label{plantstab+exo}
  \ba{rcll}
  \dot x &=& f(x, y) & \quad x \in {\cal C}_\dimx \subset \Real^{\dimmu + \dimx}\\
  \dot y &=&  \kappa\, A y +  B(q(x,y) + v)  & \quad y
  \in  \Real^{\dimy}
  \ea
 }
 and system (\ref{zd1}) as $\dot x = f(x,0)$.

 The existence of a locally Lipschitz regulator solving the problem at hand, will be claimed under
 an assumption which involves the ability of asymptotically reproducing
 the function $q(x(t),0)$, where $x(t)$ is any solution of $\dot x
 = f(x,0)$ which can be generated by taking initial conditions on $\ca$, by means of a locally
 Lipschitz system properly defined. The following definition
 aims to formally state the required reproducibility conditions which will be then used
 in the forthcoming Theorem \ref{TheoremFirst}.

 \begin{definition} (LER, rLER). \label{DefinitionRegularitySoft}
 A triplet $(F(\cdot), Q(\cdot), {\cal A})$, where
 $F: \Real^\dimz \to \Real^\dimz$ and $Q: \Real^\dimz \to \Real$
 are smooth functions and ${\cal A}\subset \Real^\dimz$ is a compact
 set, is said to be {\em Locally Exponentially Reproducible} (LER), if
 there exists a compact set ${\cal R} \supseteq \cal A$ which is LES
 for $\dot z = F(z)$ and, for any bounded set $\cal Z$ contained in
 the domain of attraction of $\cal R$, there exist an integer $\dimxi$,
 {\em locally Lipschitz} functions $\varphi: \Real^\dimxi \to \Real^\dimxi$, $\gamma:
 \Real^\dimxi \to \Real$, and
 $\psi: \Real^\dimxi \to \Real^\dimxi$, with $\psi$ a complete vector field, and a {\em locally Lipschitz} function
 $T:\Real^\dimz \to \Real^\dimxi$,
 such that
   \beeq{\label{Timmdef}
    \ba{rcl}
       Q(z) + \gamma(T(z)) =0
    \ea \qquad \forall \; z \in {\cal R}\,,
   }
   and for all $\xi_0 \in \Real^\dimxi$ and
   $z_0 \in \cal Z$  the solution $(\xi(t), z(t))$ of

       \beeq{\label{xisysdef}
        \ba{rcll}
         \dot z &=& F(z) & z(0) = z_0\\
            \dot \xi &=& \varphi(\xi) +
            \psi(\xi)\, Q(z)  & \xi(0) = \xi_0
   \ea
     }
     satisfies
     \beeq{\label{condstabdef}
        |(\xi(t),z(t))|_{{\rm graph}\left. T \right |_{\cal R}}
        \leq \beta(t, |(\xi_0, z_0)|_{{\rm graph}\left . T \right |_{\cal R}
        })
     }
    where $\beta(\cdot,\cdot)$ is a locally exponentially class-$\cal
    KL$ function.

    Furthermore the triplet in question is said to be {\em robustly Locally Exponentially
    Reproducible} (rLER) if it is LER and, in addition, for all locally essentially bounded $v(t)$, for all $\xi_0 \in \Real^\dimxi$ and
   $z_0 \in \cal Z$  the solution $(\xi(t), z(t))$ of

       \beeq{\label{xisysdefRob}
        \ba{rcll}
         \dot z &=& F(z) & z(0) = z_0\\
            \dot \xi &=& \varphi(\xi) +
            \psi(\xi)[Q(z) + v(t)] & \xi(0) = \xi_0
   \ea
     }
     satisfies
     \beeq{\label{condstabdefRob}
        |(\xi(t),z(t))|_{{\rm graph}\left. T \right |_{\cal R}}
        \leq \beta(t, |(\xi_0, z_0)|_{{\rm graph}\left . T \right |_{\cal R}
        })+ \ell(\sup_{\tau \leq t} |v(\tau)|)
     }
    where $\beta(\cdot,\cdot)$ is a locally exponentially class-$\cal
    KL$ function and $\ell$ is a class-$\cal K$ function.
     $\triangleleft$
 \end{definition}

 We postpone to Section \ref{SecDigression} a broad discussion about this
 definition and to Section \ref{SecSuff} the presentation of
 sufficient conditions for a triplet to be rLER.

%
%
%
%

 With this definition at hand, we pass to formulate the following
 theorem which fixes a framework where the stabilization problem
 previously formulated can be solved by means of a locally
 Lipschitz regulator. The proof of this theorem can be found
 in Appendix \ref{ProofTheoremFirst}.\\[1mm]
 \begin{theorem} \label{TheoremFirst}
  Let $\ca$ be LAS(${\cal X}$) for the system $\dot x = f(x,0)$ for some compact set ${\cal X} \subset {\cal C}_n$.
  Assume, in addition, that the triplet
  $(f(x,0),{q}(x,0),{\ca})$ is {\em LER}.
  Then there exist a locally Lipschitz regulator of the form
  (\ref{etasys}), a compact set ${\cal R} \supseteq \ca$, a continuous function $\tau: {\cal R} \to \Real^\dimeta$,
  and, for any  compact set ${\cal Y} \subset  \Real^\dimy$ and ${\cal N} \subset \Real^\dimeta$, a positive
  constant $\kappa^\star$, such that for all $\kappa \geq \kappa^\star$ the set
 \beeq{\label{graphtau}
   {\rm graph }\tau \times \{0\} = \{(\eta, x,y) \in \Real^\dimeta \times {\cal R} \times \Real^\dimy
   \quad : \quad \eta = \tau(x)\,, \; y=0\}
 }
 is LES(${\cal N} \times {\cal X} \times {\cal Y}$) for (\ref{plantstab+exo}), (\ref{etasys}) and the set
 \beeq{\label{graphtau}
   {\rm graph }\left . \tau \right |_\ca \times \{0\} = \{(\eta, x,y) \in \Real^\dimeta \times {\ca} \times \Real^\dimy
   \quad : \quad \eta = \tau(x)\,, \; y=0\}
 }
  is LAS(${\cal N} \times {\cal X} \times {\cal  Y}$) for (\ref{plantstab+exo}),
  (\ref{etasys}). Furthermore, if $\ca$ is also LES for the system $\dot x =
  f(x,0)$, the set $\cal R$  can be taken equal to
  $\ca$.
  \end{theorem}

 \begin{remark}
 By going throughout the proof of the previous theorem, it turns out that the
 regulator (\ref{etasys}) solving the problem at hand has the
 form
 \[
  \ba{rcl}
  \dot \eta &=& \varphi(\eta) - \psi(\eta)[\gamma(\eta) + \kappa B\tr A
  y]\\
  v &=& \gamma(\eta)
  \ea
 \]
 in which $\kappa$ is a sufficiently large positive number and
  $(\varphi(\cdot),\psi(\cdot),\gamma(\cdot))$ are the locally
  Lipschitz functions which are associated to the triplet
  $(f(x,0),{q}(x,0),{\ca})$ in the definition of local exponential
  reproducibility. $\triangleleft$
 \end{remark}

\section{A brief digression about the problem} \label{SecDigression}
 The structure of (\ref{plantstab}) and the associated
 problem, apparently very specific, are indeed recurrent in a number of control scenarios
 in which robust non  linear stabilization is involved. We refer to Section \ref{SecUCO} for
 the presentation of a few relevant cases where this
 occurs. For the time being it is interesting to note how the previous formulation presents two
 main peculiarities which make the problem at hand particularly challenging.\\
  The first is that the function $q(w,x,y)$, coupling the $x$ and $y$ subsystem in
 (\ref{plantstab}), is not necessarily vanishing on the desired
 attractor $\ca \times \{0\}$, namely the desired
 attractor $\ca \times \{0\}$ is not necessarily forward invariant
 for (\ref{plantstab}) in the case $v\equiv 0$. In this respect the first
 crucial property required to the regulator is to be able to reproduce, through the
 input $v$, the uncertain coupling term
 $q(w,x,0)$ by providing a not necessarily zero steady-state
 control input. This issue is intimately connected to arguments which are usually
 addressed in the output regulation literature (see \cite{SICON}, \cite{BI03}), \cite{Huang},
 \cite{Pavlov}),  in which the goal is precisely to make attractive a set, on which regulation objectives are met,
 which is not invariant for the open-loop system.

 The second peculiarity, apparently not correlated to the previous one,
 relies in the fact that the set $\ca$ is assumed to be "only" asymptotically stable for
 (\ref{zd1}) and no exponential properties are required. In this
 respect the study of the interconnection (\ref{plantstab}) is particularly
 challenging as it is not sufficient, in general, to decrease the linear
 asymptotic gain (\cite{Te96}) between the "inputs" $x$ and the "outputs" $y$ of the $y$-subsystem
 (which is what one would make by increasing the value of $\kappa$ since the matrix $A$
 is Hurwitz) to infer asymptotic properties in the interconnection.
 Indeed the presence of a not necessarily linear asymptotic gain
 between the "inputs" $y$ and the "outputs" $x$ of the
 $x$-subsystem requires a non trivial design of the input $v$
 which, intuitively, should be chosen to infer a certain locally non-Lipschitz
 ISS gain to the $y$-subsystem.

  The rich available literature on nonlinear stabilization already provides
  successful tools to solve the problem at hand if the previous two pathologies
  are dropped, namely if the  assumption is strengthen
  by asking that the set $\cal A$ is also LES(${\cal X}$) for (\ref{zd1}) and
  that the "coupling" term $q(w,x,y)$  is vanishing at ${\cal A} \times \{ 0\}$.
  As a matter of fact, under the previous conditions, it is a well-known fact
  that the set $\ca \times \{0\}$, which is forward invariant for (\ref{plantstab})
  with $v=0$, can be stabilized by means of a large value of $k$ as
  formalized in  (\cite{TP}, \cite{Atassi}).\\

 In the case $\ca$ is not exponentially stable
 for (\ref{zd1}) and/or the  coupling term $q(w,x,y)$ is not vanishing on
 the desired attractor,  the problem becomes challenging and more sophisticated
 choices for $v$ must be envisaged. In particular, while preserving the local Lipschitz property of the
 regulator,
 the only conclusions which can be drawn if $v\equiv 0$ is that the
 origin is {\em semiglobally practically} stable in the parameter
 $\kappa$, that is the trajectories of the system can be steered arbitrary
 close to the set $\ca \times \{0\}$ by increasing the value of $\kappa$ (see \cite{TP}, \cite{Atassi}, \cite{SICON}).
 Even in the simpler scenario in which $q(w,x,0)\equiv 0$ for
 all $(w,x) \in \ca$, a large value of  $k$ is not sufficient
 to enforce the desired asymptotic behavior in the case the set
 $\ca$ fails to be exponentially stable for (\ref{zd1}).
  In this case the asymptotic properties of the system
  have been studied in \cite{Celani}
 by showing how the trajectories are attracted by
 a manifold which, only in a particular case depending on the linear approximation
 of the system, collapses to the origin (see Theorem 6.2 in  \cite{Celani}).

  In these critical scenarios an appropriate design of the control input $v$
  becomes inevitable in order to compensate for the coupling term $q(w,x,y)$  which cannot
  be only dominated by a large value of $\kappa$.
  In particular, a first possible option, motivated by small
  gain arguments and gain assignment procedures for nonlinear systems (see \cite{JiangBook},
  \cite{JiangTeelPraly}), is to design the control $v$ in order  to
  assign, to the $y$-subsystem, a certain nonlinear ISS gain suitably identified according to small gain criterions
  and to the asymptotic gain of the $x$-subsystem (\ref{plantstab}).
  This option, however, necessarily leads to design control laws
  which are not, in general, locally Lipschitz close to the
  compact attractor and, thus, which violates a basic requirement of the above
  problem.\\
  An alternative option to design the control $v$ is
  to be inspired by nonlinear separation principles (see, besides others,
  \cite{TP}, \cite{TPSCL}, \cite{Atassi}, \cite{IsBook2}, \cite{GKbook}), namely  to
  design an appropriate state observer yielding an asymptotic estimate
  $(\hat w,\hat x,\hat y)$ of the state variables, and to
  asymptotically compensate for the coupling term $q(w,x,y)$ by
  implementing a ``certainty equivalence'' control law of the form $v = - q(\hat w,\hat x, \hat
  y)$. Indeed, under suitable conditions, the tools proposed in
  \cite{TPSCL} would allow one to precisely fix the details and to solve the problem
 at hand in a rigorous way.
 This way of approaching the problem, though, presents a number of
 drawbacks which substantially limit its applicability.
 First, the design of the observer clearly requires the formulation of
 suitable observability assumptions\footnote{It must be noted that
 only {\em local} observability notion are potentially needed at this level as a consequence
 of the fact that practical stability is already guaranteed by the high-gain law
 $\kappa$.} on the controlled plant, and in particular of its $(w,x)$ components, not in principle necessary
 for the stabilization problem to be solvable, which may be not fulfilled
 for a number of relevant cases.
 Moreover, according to the state-of-the-art of the observer
 design literature (\cite{GKbook}), the design of the observer may be a challenging (if not impossible)
 task in case of uncertain parameters affecting the observed
 dynamics. Finally, it is worth noting how approaching the problem according to the
 previous design philosophy, leads to inherently {\em redundant control structures}, by
 requiring the explicit estimate of the full state (and of possible uncertainties)
 in order to reproduce the signal  $q(w,x,y)$.

 As opposite to the previous strategies, Theorem \ref{TheoremFirst}
 provides a design procedure which does not rely upon domination of the interconnection term
 $q(w,x,y)$ but rather on its asymptotic reconstruction which, however, it is not based upon the design of
 an observer of the state variables $(w,x,y)$.
 In this respect the crucial
 property underlying the Theorem is the local exponential
 reproducibility property which, according to its definition, relies upon two requirements.
 The key first requirement, for a triplet $(F,Q,\ca)$ to be
 LER,  is that there exists a set $\cal R$ which contains $\ca$
 and which is LES for the autonomous system $\dot z = F(z)$.
  The second crucial requirement characterizing the
  definition is that there exists a {\em locally Lipschitz} system of the form
  \beeq{\label{xisys}
 \ba{rcl}
  \dot \xi &=& \varphi(\xi) +  \psi(\xi) u_\xi\\
  y_\xi &=& \gamma(\xi)
  \ea
  }
  with input $u_\xi$ and output $y_\xi$,
 such that system (\ref{xisysdef}), modelling the cascade connection of the autonomous system $\dot z = F(z)$
 with output $y_z = Q(z)$ with the system (\ref{xisys}), has a
 locally exponentially stable set described by  ${\rm graph} \left .
 T \right |_{\cal R}$ and, on this set, the output $y_\xi$ equals $y_z$ (see (\ref{Timmdef})).
 The domain of attraction of ${\rm graph} \left .
 T \right |_{\cal R}$ is required to be of the form ${\cal Z} \times \Real^\dimxi$ with $\cal Z$ any
 compact set in the domain of attraction of $\cal R$ (note that, according to the definition, system
 (\ref{xisys}) is allowed to depend on the choice of $\cal Z$).
 In this respect the second requirement can be regarded as the ability, of the system
 (\ref{xisys}), of {\em asymptotically reproducing} the output function $Q(z(t))$ of system $\dot z(t) = F(z(t))$ with initial
 conditions of the latter taken in $\cal Z$. Note how the "output reproducibility" property
 required to system (\ref{xisys}) does not hide, in principle, any
 kind of state observability property of the system $\dot z=F(z)$
 with output $y_z= Q(z)$. In other words system (\ref{xisys}) must
 be not confused with a state observer of the $z$-subsystem as its role is to reproduce
 the output function $Q(z(t))$ and not necessarily to estimate its state.

  As the definition of {\em robust} LER, we only note that, in addition to the previous properties, it is required
  that system (\ref{xisysdefRob}) exhibits an ISS property (without any special requirement on the asymptotic gain)
  with respect to the exogenous input $v$.

 \section{Applications}
\subsection{Output-feedback from UCO state-feedback in presence of non-hyperbolic
attractors}\label{SecUCO}

 In this part we show how the theory of robust nonlinear separation principle presented in
 \cite{TP}, \cite{Atassi}
  can be extended with the tools developed in the previous sections. In particular we are interested to
  extend the theory of \cite{TP} by showing how to design a pure
  output-feedback semiglobal controller stabilizing an attractor when it is known how the
  latter can be asymptotically ({\em but not exponentially}) stabilized
  by means of a  Uniform Completely Observable (UCO) state-feedback
  controller.

 Consider the smooth system
 \beeq{\label{pla0withw}
  \ba{rcll}
  \dot w &=& s(w) & \qquad w \in  W \subset \Real^\dimmu\\
  \dot z &=& A(w,z,u) & \qquad z \in \Real^\dimz, \, u \in \Real\\
  y &=& C(w, z) & \qquad y \in \Real
  \ea
 }
 in which $u$ and $y$ are respectively the control input and the
 measured output and $W$ is a compact set which is invariant for
 $\dot w=s(w)$. As discussed in the previous section, the variables  $w$ emphasize the possible presence
 of parametric uncertainties and/or disturbance to be rejected and/or reference to be tracked (in the latter case the
 measurable output $y$ plays more likely the role of regulation/tracking error).
 As done before, in order to simplify the notation, we drop the
 dependence of the variable $w$ and we compact system
 (\ref{pla0withw}) in the more convenient form

  \beeq{\label{pla0}
  \ba{rcll}
  \dot z &=& A(z,u) & \qquad z \in \Real^\dimz, \, u \in \Real\\
  y &=& C(z) & \qquad y \in \Real
  \ea
  }
  which is supposed to evolve on a closed invariant set ${\cal C}_\dimz$ which is endowed with the subset topology
 (such a closed set being, in the form (\ref{pla0withw}), the closed cylinder ${\cal C}_\dimz:= W \times
  \Real^\dimz$).

 We recall (see \cite{TP}) that a function $\bar u :  \Real^\dimz
 \to \Real$ is said to be UCO with respect to (\ref{pla0}) if there exist two integers
 $n_y$, $n_u$ and a ${\cal C}^1$ function $\Psi$ such that, for each solution of
 \beeq{\label{plaextended}
  \ba{rcl}
  \dot z &=& A(z,u_0)\\
  \dot u_i &=& u_{i+1} \qquad i=0,\ldots, n_u-1\\
  \dot u_{n_u} &=& v\\
 \ea
 }
 we have, for all $t$ where the solution makes sense,
 \beeq{ \label{ubarnot}
 \bar u(z(t)) = \Psi(y(t), y^{(1)}(t), \ldots, y^{(n_y)}(t),
u_0(t), \ldots, u_{n_u}(t) )
 }
 where $y^{(i)}(t)$ denotes the $i$th derivative of $y$ at time
 $t$.

  Motivated by \cite{TP} we shall study system (\ref{pla0}) under
  the following two assumptions:

   \begin{itemize}
  \item[a)] there exist a smooth function $\bar u :  \Real^\dimz
 \to \Real$ and compact sets $\ca \subset {\cal C}_\dimz$ and ${\cal Z} \subset {\cal
 C}_\dimz$, such that the $\ca$ is LAS($\cal Z$) for system (\ref{pla0}) with $u=  \bar u(z)$;\footnote{By referring
 to (\ref{pla0withw}), a meaningful case  to be considered
  is when $\ca = W \times \{0\}$, in which case this assumption amounts
  to require the existence  of a state feedback stabilizer, possibly dependent on
  the uncertainties, able to asymptotically stabilize the origin with a certain domain of attraction.}
  \item[b)] $\bar u(z)$ is UCO with respect to (\ref{pla0}).
   \end{itemize}

 In this framework we shall be able to prove, under suitable
 reproducibility conditions specified later, that the previous two
 assumptions imply the existence of a locally Lipschitz dynamic output
 feedback regulator able to asymptotically stabilize the set
 $\ca$. The main theorem in this direction is detailed next. In
 this theorem we refer to an integer $\ell_u \geq n_u$ defined as
 that number such that for the system

 \beeq{\label{extendedzsys}
       \ba{rcl}
        \dot z &=& A(z,\xi_0)\\
         \dot \xi_0 &=& \xi_1\\
         &\vdots&\\
         \dot \xi_{\ell_u} &=& u_1\,,
         \ea
        }
 there exist smooth functions $C_i$ such that the first $n_y+1$
 time derivatives of $y$ can be expressed as
  \[
    y^{(i)} = C_i(z,\xi_0,\ldots, \xi_{\ell_u}) \qquad \forall \,
   i=0, \ldots, n_y +1\,.
  \]

 \begin{theorem} \label{TheoremMainUCO}
 Consider system (\ref{pla0}) and assume the existence of a
 compact set $\ca \subset {\cal C}_\dimz$ and of a smooth
 function $\bar u(z)$ such that properties (a) and (b) specified
 above are satisfied. Assume, in addition,  that  the triplets
 \beeq{ \label{triplet1}
 (A(z,\bar u(z)), L_{A(z,\bar u(z))}^{(\ell_u +1)} \bar u(z), {\ca})
 }
 and
 \beeq{ \label{triplet2}
 (A(z,\bar u(z)), L_{A(z,\bar u(z))}^{(n_y +1)}  C(z), {\ca})
 }
 are rLER. Then
 there exist a positive $\dimzeta$, a compact set ${\cal B} \subset \Real^\dimzeta$ and,
 for any ${\cal N} \subset \Real^\dimzeta$, a locally Lipschitz controller of the form
 \beeq{\label{finalregUCO}
 \ba{rcl}
 \dot \zeta &=& \Phi(\zeta,y) \qquad \zeta \in \Real^\dimzeta\\
 u &=& \Upsilon(\zeta,y)
 \ea
 }
 such that the set $\ca \times \cal B$ is LAS(${\cal Z} \times
 \cal N$) for the closed-loop system (\ref{pla0}),
 (\ref{finalregUCO}).
 \end{theorem}
\vspace{1mm}

  This result extends Theorem 1.1 of \cite{TP} in
  three directions.
  First, note that we are dealing with stabilization of compact attractors for
  systems evolving on {\em closed sets}.
  This is a technical improvement on which, though,  we would not like to put the
  emphasis.
  Second, note that the
  UCO control law $\bar u(z)$ is not required to be vanishing on the attractor
  $\ca$ which, as a consequence, is not required to be forward invariant for the
  open loop system (\ref{pla0}) with $u\equiv 0$. In this respect the proposed setting can be seen as also
  able to frame output regulation problems.
   Finally, the
  previous result claims  that, by means
  of a pure {\em locally Lipschitz} output feedback controller, we are able to restore the
  asymptotic properties of an UCO controller without relying upon
  exponential stability requirements of the latter and robustly with respect
  to uncertain parameters. The last two extensions are conceptually
  very much relevant and can be seen as particular application of the tools
  presented in the previous sections.
  Following the main laying of \cite{TP}, the proof of the claim is divided in
  two subsections which contain results interesting on their own.

\subsection{Robust Asymptotic Backstepping} \label{SectionRobAsyBac}

 In this part we discuss how the UCO control law $\bar u$ can be
 robustly back-step through  the chain of integrators of
 (\ref{plaextended}). As commented above, the forthcoming
 proposition extends in a not trivial way the results of \cite{TP} in the
 measure in which one considers the fact that $\bar u(z)$
 is not vanishing on the attractor and that $\ca$ is not necessarily
 locally exponential stable for the closed-loop system.

  We show that the existence of the static UCO stabilizer for (\ref{pla0})
 implies the existence of a dynamic stabilizer for (\ref{extendedzsys}) using the
 partial state $\xi_i$, $i=0,\ldots, \ell_u$, and the output derivatives $y^{(i)}$,
 $i=1,\ldots, n_y$. This is formally proved in the next
 proposition.

\begin{proposition}\label{theorembackstepping}
  Consider system (\ref{extendedzsys}) under the assumption (a) previously formulated.
 Assume that the triplet (\ref{triplet1}) is rLER.
 Then there exists a positive $\dimeta$,  a
 compact set ${\cal R} \supset \ca$, a continuous function
 $\tau: {\cal R} \to \Real^{\dimeta + \ell_u+1}$, and,
 for any compact set $\Xi' \subset \Real^{\ell_u+1}$
 and ${\cal N}' \subset \Real^\dimeta$, a locally Lipschitz regulator of the form
 \beeq{\label{reg1}
 \ba{rcl}
  \dot \eta &=& \varphi(\eta, \xi, \bar u(z)) \qquad \eta \in
  \Real^{\dimeta}\\
  \qquad  u_1 &=& \rho(\eta, \xi, \bar u(z))\,,
  \ea
 }
 with $\xi={\rm col}(\xi_0,\ldots, \xi_{\ell_u})$ such that the sets
  \beeq{\label{asyset}
  {\rm graph} \, \tau :=
  \{ (z,\xi,\eta) \in {\cal R} \times \Real^{\ell_u+1}
  \times \Real^{\dimeta}  \; :
  \; (\xi,\eta) = \tau(z)\}
  }
 and
 \[
 {\rm graph} \left . \tau \right |_{\ca}
 :=
  \{ (z,\xi,\eta) \in {\ca} \times \Real^{\ell_u+1}
  \times \Real^{\dimeta}  \; :(\xi,\eta) = \tau(z)\}
  \]
  are respectively LES(${\cal Z} \times {\Xi}' \times {\cal N}'$) and
 LAS(${\cal Z} \times {\Xi}' \times {\cal N}'$) for the closed-loop system
 (\ref{extendedzsys}), (\ref{reg1}).
\end{proposition}

\begin{proof}
 Consider the change of variables
 \[
 \ba{rcl}
 \xi_0 \; \to \tilde \xi_0 &:=& \xi_0 - \bar u(z)\\
 \xi_i \; \to \; \tilde \xi_i &:=& \xi_i -\bar u^{(i)}(z,\tilde \xi_0, \ldots, \tilde \xi_{i-1})
\quad i=1, \ldots, \ell_u
 \ea
 \]
 where the $\bar u^{(i)}(z)$, $i=1, \ldots, \ell_u$, are recursively defined as
 \[
 \ba{rcl}
 \bar u^{(1)}(z,\tilde \xi_0) &:=& \dst {\partial \bar u(z) \over \partial z} A(z,\tilde \xi_0 + \bar
 u(z))\\
 \bar u^{(i)}(z, \tilde \xi_0, \ldots, \tilde \xi_{i-1} ) &:=& \dst {\partial \bar u^{(i-1)}(z,\tilde \xi_0,\ldots
 \tilde \xi_{i-2}) \over \partial z} A(z,\tilde \xi_0 + \bar
 u(z)) + \\
 && \dst \sum_{j=0}^{i-2} {\partial \bar u^{(i-1)}(z,\tilde \xi_0,\ldots
 \tilde \xi_{i-2}) \over \partial \tilde \xi_j}\tilde \xi_{j+1} \quad\qquad
 i=2,\ldots,\ell_u\,,
 \ea
 \]
 and the further change of variable
  \[
  \ba{l}
  \tilde  \xi_i,\; \to \;  \zeta_i:=  g^{-i} \tilde \xi_i \qquad i=0,\ldots,
  \ell_u-1\\
  \tilde \xi_{\ell_u} \; \to \; \zeta_{\ell_u} := \tilde \xi_{\ell_u} - \dst \sum_{i=0}^{\ell_u -1}
  a_i g^{\ell_u-i} \tilde \xi_i
 \ea
  \]
  where $g$ is a positive design parameter and the $a_i$'s
  are coefficients of an Hurwitz polynomial.

  By letting $\zeta := \mbox{col} (\zeta_0, \ldots,
 \zeta_{\ell_u-1})$ system (\ref{extendedzsys}) in the new coordinates reads as
  \beeq{\label{syszxizeta}
  \ba{rcl}
   \dot z &=& A(z,\bar u(z)) + \tilde A(z,C\zeta)\\
   \dot {\zeta} &=& g H {\zeta} + B \zeta_{\ell_u}\\
   \dot \zeta_{\ell_u} &=& u_1 + \ell_g(z,\zeta, \zeta_{\ell_u})
     \ea
  }
  where $B=\mbox{col}(0,\ldots, 0,1)$, $C=(1,0,\ldots,0)$, $\tilde A(z,C\zeta) = A(z,\tilde \xi_0 + \bar u(z))
  -A(z,\bar u(z))$, $H$ is a Hurwitz matrix
  and $\ell_g(\cdot)$ is a smooth function such that
  \beeq{\label{ellprop}
  \ell_g(z,0,0)=-L_{A(z,\bar u(z))}^{(\ell_u+1)} \bar u(z) \qquad \forall \, z \in \Real^n\,.
  }
 As the triplet (\ref{triplet1}) is rLER, there exists a compact
 set ${\cal R} \supseteq \ca$ which is LES for $\dot z = A(z,\bar u(z))$
 with ${\cal D}({\cal R}) \supseteq {\cal D}(\ca)$. Furthermore, by the fact that (\ref{triplet1}) is rLER and
 by definition of rLER, also the triplet $(A(z,\bar u(z)), -L_{A(z,\bar
u(z))}^{(\ell_u+1)} \bar u(z), {\cal R})$ is rLER.
  We consider now the zero dynamics, with respect to the input $u_1$
 and output $\zeta_{\ell_u}$, of system (\ref{syszxizeta}) given
 by
 \beeq{\label{zdclosedloopproof1}
 \ba{rcl}
   \dot z &=& A(z,\bar u(z)) + \tilde A(z,C\zeta)\\
   \dot {\zeta} &=& g H {\zeta}
   \ea
 }
 For this system it can be proved (by means of arguments which, for instance, can be found in
 \cite{SICON}), that for any compact set ${\cal M} \in
 \Real^{\ell_u}$ there exists a $g^\star >0$ such that for all $g \geq
 g^\star$ the sets ${\cal R} \times \{ 0\}$ and $\ca \times \{0\}$ are
 respectively LES(${\cal Z} \times {\cal
 M}$) and  LAS(${\cal Z} \times {\cal
 M}$) for (\ref{zdclosedloopproof1}).   Fix, once for
all, $g \geq  g^\star$. By the previous facts, by (\ref{ellprop}),
by the fact that the triplet $(A(z,\bar u(z)), -L_{A(z,\bar
u(z))}^{(\ell_u+1)} \bar u(z), {\cal R})$ is rLER, and by
Proposition \ref{Propx1x2} in Appendix \ref{AppendixAuxRes},
 it follows that the triplet $ ((\ref{zdclosedloopproof1}),\, \ell_g(z,\zeta,
 0),\, {\cal R} \times \{ 0\})$ is LER. Now fix
  \beeq{\label{controlu}
 u_1 = - \kappa \zeta_{\ell_u} + v
 }
 where $\kappa$ is a positive design parameters and $v$ is a
 residual control input.
 From the previous results, it follows that
 system (\ref{syszxizeta}) with (\ref{controlu}) fits in the framework of
 Theorem \ref{TheoremFirst} , by which it is possible to conclude that there exists a locally Lipschitz
 controller of the form
 \beeq{\label{etasysprime}
  \dot \xi = \Phi_k'(\xi, \zeta_{\ell_u}) \qquad v=\Upsilon_k'(\xi,\zeta_{\ell_u})
  \qquad \xi \in \Real^\dimxi\,,
  }
  a continuous function $\tau': {\cal R}\times \{0\} \to \Real^{\dimxi}$
  and, for any compact set ${\cal M}_{\ell_u} \subset
  \Real$ and ${\cal N}' \subset \Real^{\dimxi}$, a positive
  constant $\kappa^\star$, such that for all $\kappa \geq
  \kappa^\star$ the set
  \[
 {\rm graph }\tau' \times \{0\} = \{((z,\zeta),\zeta_{\ell_u},\xi) \in
  ({\cal R} \times \{0\}) \times \Real \times \Real^{\dimxi} \quad : \quad
  \zeta_{\ell_u}=0\,, \; \xi = \tau'(z,\zeta)\}
   \]
 is LES(${\cal Z} \times {\cal M} \times {\cal M}_{\ell_u} \times
{\cal N}'$) for (\ref{syszxizeta}), (\ref{controlu}) and
  (\ref{etasysprime}).
 Furthermore, by properly adapting the arguments at the end of the
 proof of Theorem \ref{TheoremFirst}, it is possible also to prove
 that the set ${\rm graph }\left . \tau'\right |_{\ca \times \{0\}} \times
  \{0\}$ is LAS(${\cal Z} \times {\cal M} \times {\cal M}_{\ell_u} \times {\cal N}'$)
  for (\ref{syszxizeta}), (\ref{controlu}) and
  (\ref{etasysprime}).

 The previous facts have shown how to solve the problem at hand by
 means of a {\em partial state feedback} regulator (namely a regulator processing
 $\tilde \xi_0 = \xi_0 - \bar u(z)$  and its first $\ell_u$ time derivatives $\tilde \xi_i$) of the form
 \beeq{\label{partialstatefeedbackreg}
 \ba{rcl}
  \dot \xi &=& \Phi'_k(\xi, \zeta_{\ell_u})\\
  u_1 &=& -\kappa \zeta_{\ell_u} + \Upsilon'_k(\xi, \zeta_{\ell_u}) \qquad
  \zeta_{\ell_u} = \tilde \xi_{\ell_u} - \sum_{i=0}^{\ell_u-1} a_i
  g^{\ell_u-1} \tilde \xi_i \,.
  \ea
 }
 In order to obtain a pure output feedback regulator of the form
 (\ref{reg1}), we follow \cite{TP} and design a "dirty derivatives"
 observer-based  regulator
 \beeq{\label{dirtyderReg}
  \ba{rcl}
  \dot {\hat {\tilde \xi}}_{i} &=& \hat {\tilde \xi}_{i+1} + L^{i+1} \lambda_i (\hat {\tilde \xi}_0 -
  \tilde \xi_0)\qquad i=0, \ldots \ell_u-1\\
 \dot {\hat {\tilde \xi}}_{\ell_u} &=& L^{\ell_u+1} \lambda_r (\hat {\tilde \xi}_0 -
  \tilde \xi_0)\\[1mm]
  \dot \xi &=& \Phi'_k(\xi, \hat {\zeta}_{\rm sat})\\
  u &=& -\kappa \hat \zeta_{\rm sat} + \Upsilon'_k(\xi, \hat{\zeta}_{\rm sat})
  \ea
  }
 where $L$ is a positive design parameters, the $\lambda_i$'s are such that the polynomial
 $s^{\ell_u} + \lambda_{\ell_u} s^{\ell_u-1} + \ldots + \lambda_2 s + \lambda_1$ is
 Hurwitz and where
 \[
  \hat {\zeta}_{\rm sat} = \mbox{sat}_\ell(\hat {\tilde \xi}_{\ell_u} - g^{\ell_u-1} a_0 \hat {\tilde \xi}_1
  -
  \ldots - g
 a_{\ell_u-1} \hat {\tilde \xi}_{\ell_u-1})
 \]
 in which $\mbox{sat}_\ell(s)$ is the saturation function such
 that $\mbox{sat}_\ell(s)=s$ if $|s| \leq \ell$ and
 $\mbox{sat}_\ell(s)=\ell \, \mbox{sgn}(s)$ otherwise.
 Letting $\tilde \xi=\mbox{col}(\tilde \xi_1,\ldots, \tilde \xi_{\ell_u})$,
 $\hat {\tilde \xi} = \mbox{col}(\hat {\tilde \xi}_1,\ldots, \hat
 {\tilde \xi}_{\ell_u})$, and defining the change of variables
 \[
  \hat {\tilde \xi} \; \mapsto \; e := D_L (\tilde \xi - \hat {\tilde \xi})
 \]
 in which $D_L=\mbox{diag}(L^{\ell_u-1}, \ldots, L,1)$, it turns out
 that the overall closed-loop system (\ref{syszxizeta}),
 (\ref{dirtyderReg}) reads as
 \beeq{\label{xesysproof1}
  \ba{rcl}
  \dot x &=& \varphi(x) + \Delta_1(x,e)\\
  \dot e &=& L H e + \Delta_2(x,e)
 \ea
 }
 where
 $
  x :=  \mbox{col}(z, \zeta, \zeta_{\ell_u}, \xi)
 $,
 $\dot x = \varphi(x)$ is a compact representation of
 (\ref{syszxizeta}), (\ref{partialstatefeedbackreg}), $H$ is
 a Hurwitz matrix and $\Delta_1(\cdot)$ and $\Delta_2(\cdot)$ are
 defined as
 \[
 \Delta_1(\cdot)=\left(\ba{c}0\\0\\ \kappa(\zeta_{\ell_u} - \hat {\zeta}_{\rm sat}) +
 \Upsilon'_k(\xi,\hat {\zeta}_{\rm sat}) -  \Upsilon'_k(\xi,{\zeta}_{\ell_u})  \\
 \Phi_k'(\xi,\hat {\zeta}_{\rm sat} ) -   \Phi_k'(\xi,{\zeta}_{\ell_u} ) \ea\right )
 \]
 and
 \[
 \Delta_2(\cdot) = \left( \ba{c}
 0\\
 \vdots\\
 0\\
  \kappa(\zeta_{\ell_u} - \hat {\zeta}_{\rm sat}) +
 \Upsilon'_k(\xi,\hat {\zeta}_{\rm sat}) -  \Upsilon'_k(\xi,{\zeta}_{\ell_u})
 \ea
 \right )
 \]
  By construction the set ${\rm graph}\tau' \times \{0\}$ is
  LES($\cal X$) for the system $\dot x = \varphi(x)$
  with ${\cal X}:= ({\cal Z} \times {\cal M}) \times {\cal M}_{\ell_u} \times {\cal
  N}'$ and, by construction, it turns out that for any $\ell>0$,
  $\Delta_1(x,0)\equiv 0$ and $\Delta_2(x,0)\equiv 0$ for all $x \in {\rm graph}\tau' \times
  \{0\}$. Furthermore, for any compact ${\cal M} \in
  \Real^{\ell_u-1}$, ${\cal M}_{\ell_u} \in \Real$ and $\hat{\cal M} \in
  \Real^{\ell_u}$,  there exists a compact set ${\cal E} \subset \Real^{\ell_u}$ ({\em dependent on $L$}) such that
  if $\zeta(0) \in {\cal M}$, $\zeta_{\ell_u} \in {\cal
  M}_{\ell_u}$ and $\hat {\tilde \xi} (0) \in \hat {\cal M}$ then $e(0) \in \cal E$. Furthermore,
  by definition of saturation function, it turns out that for all $\bar x>0$ there exist $\delta_1>0$
  and $\delta_2>0$ such that  $|\Delta_1(x,e)|\leq \delta_1$ and $\Delta_2(x,e) \leq \delta_2$
  for all $x$, $|x| \leq \bar x$, $e \in \Real$ and  $L>0$.

  From these facts and by the result in \cite{TP}, it follows that for any ${\cal E} \in
  \Real^{\ell_u}$ there exists a $L^\star>0$ such that for all
  $L \geq L^\star$ the set
  \[
  {\rm graph} \tau' \times \{0\} \times \{0\}
  = \{((z, \zeta), \zeta_{\ell_u}, \xi, e) \in ({\cal R} \times \{0\})
  \times\Real \times
   \Real^\dimxi \times \Real^{\ell_u} \; : \; \zeta_{\ell_u} =0 \,, \; \xi = \tau'(z,\zeta)\,, \; e =0\}
  \]
 is LES(${\cal Z} \times {\cal M} \times {\cal M}_{\ell_u} \times {\cal
  N}' \times \cal E$).

 From the previous results, the fact that $\ca$ is LAS for the system $\dot z
 = A(z,\bar u(z))$, and the fact that on ${\rm graph} \tau' \times \{0\}
 \times \{0\}$ the closed-loop dynamics is described by
 $\dot z = A(z,\bar u(z))$), the desired result follows by
 properly adapting the omega-limit set arguments used at the end of the proof of Theorem \ref{TheoremFirst}.

\end{proof}

 \subsection{Extended Dirty Derivatives Observer}
 In this part we present  a result which allows one to obtain a
 pure output feedback stabilizer once a partial state-feedback
 stabilizer (namely a stabilizer processing the output and a certain number of its
 time derivative) is known. Along the lines pioneered in
 \cite{EK} and \cite{TP}, the idea is to substitute the
 knowledge of the time derivatives of the output with
 appropriate estimates provided by a "dirty derivative observer" (by using the terminology of \cite{TP}).
 In our context, though, we propose an "extended" dirty derivative
 observer, where the adjective "extended" is to emphasize the
 presence of a dynamic extension of the classical observer
 structure motivated by the need of handling the presence of
 possible not exponentially stable attractors in the partial-state feedback
 loop and the fact that, on this attractor, the measured output is not necessarily vanishing.

 More specifically we assume, for the system (\ref{pla0}), the existence of a dynamic stabilizer of
 the form
  \beeq{\label{varsigmareg}
  \ba{rcl}
  \dot \varsigma &=& \bar \varphi(\varsigma, y,  y^{(1)}, \ldots, y^{(n_y)}) \quad \varsigma \in \Real^\dimvarsigma\\
  u &=& \bar \rho(\varsigma, y, y^{(1)}, \ldots, y^{(n_y)})
  \ea
  }
  such that the following property hold for the closed-loop system:

  \begin{itemize}
  \item[a)]  there exists a compact set ${\cal R} \supset \ca$ and a
 continuous function $\tau: {\cal R} \to \Real^\dimvarsigma$ such
 that the sets ${\rm graph} \tau$ and ${\rm graph} \left . \tau \right
|_{\ca}$ are respectively LES(${\cal Z}  \times {\cal H}$) and LAS(${\cal Z}  \times {\cal H}$) for
the closed-loop system (\ref{pla0}), (\ref{varsigmareg}) for some compact set ${\cal H} \subset
\Real^\dimvarsigma$;

 \item[b)] there exist smooth functions $C_i$, $i=0,\ldots, n_y+1$, such that
 the output derivatives $y^{(i)}$ of the closed-loop system (\ref{pla0}),
 (\ref{varsigmareg}) can be expressed as $y^{(i)} =
 C_i(z,\varsigma)$, $i=0,\ldots, n_y+1$ and the following holds
 \[ \left . \bar \rho(\varsigma, y, y^{(1)}, \ldots, y^{(n_y)})
\right |_{{\rm graph} \tau} =\bar u(z)\,.
\]

 \end{itemize}

 \begin{remark}
 Note that the previous conditions are automatically satisfied under the
 assumptions of Section \ref{SecUCO} and by virtue of the results presented in the previous
 section. As a matter of fact, by bearing in mind (\ref{extendedzsys}) and Theorem
 \ref{theorembackstepping} (and specifically (\ref{reg1})), the
 main outcome of the previous Section has been to design a
 dynamic controller of the form
 \beeq{\label{reggg}
 \ba{rcl}
 \dot \xi_0 &=& \xi_1\\
         &\vdots&\\
         \dot \xi_{\ell_u} &=& \rho(\eta, \xi, \bar u(z))\\
 \dot \eta &=& \varphi(\eta, \xi, \bar u(z)) \\[1mm]
 u &=& \xi_0
 \ea
 }
 in which, according to (\ref{ubarnot}) and to the definition of
$\ell_u$,
 \beeq{\label{ubarxi}
 \bar u(z) = \Psi(y, y^{(1)}, \ldots, y^{(n_y)}, \xi_0, \ldots,
 \xi_{n_y})\,.
 }
 System (\ref{reggg}), (\ref{ubarxi}) is clearly in the form (\ref{varsigmareg}) and, according to the result
 of Theorem \ref{theorembackstepping}, the previous conditions (a)-(b) are satisfied. $\triangleleft$
 \end{remark}

 Within the previous framework we are able to prove the following
proposition which, along with Proposition \ref{theorembackstepping} and the previous remark,
immediately yields Theorem \ref{TheoremMainUCO}.\\[1mm]

 \begin{proposition}
 Consider system (\ref{pla0}) and assume the existence of a
dynamic regulator of the form (\ref{varsigmareg}) such that the
previous properties (a)-(b) are satisfied. Assume, in addition,
that the triplet (\ref{triplet2}) is rLER. Then
 there exist a positive $\dimzeta$, a compact set ${\cal B} \subset \Real^\dimzeta$ and,
 for any compact set ${\cal N} \subset \Real^\dimzeta$, an output
 feedback controller of the form (\ref{finalregUCO}) such that the set $\ca \times \cal B$ is
LAS(${\cal Z} \times \cal N$) for the closed-loop system
(\ref{pla0}), (\ref{finalregUCO}).
 \end{proposition}
\vspace{2mm}

\begin{proof}
 As candidate controller, we consider a system
of the form
\[
   \ba{rcl}
  \dot \varsigma &=& \bar \varphi_\ell(\varsigma, y, \hat y_1, \ldots, \hat
y_{n_y})\\[2mm]
  \dot {\hat y}_i &=& \hat y_{i+1} + L^{i+1} \lambda_i (\hat
  y_0- y)\\
  \dot {\hat y}_{n_y} &=& L^{n_y +1} \lambda_{n_y} (\hat
y_0-y) + v\\[2mm]
   u &=& \bar \rho_\ell(\varsigma, y, \hat y_1, \ldots, \hat y_{n_y})\\
  \ea
\]
in which $v$ is a control input to be designed, $L$ is a positive
design parameters, the $\lambda_i$'s are the coefficients of an
Hurwitz polynomial and $\bar \varphi_\ell(\cdot)$ and $\bar
\rho_\ell$ are appropriate saturated versions of the functions
$\bar \varphi(\cdot)$ and $\bar \rho(\cdot)$ of
(\ref{varsigmareg}) satisfying $ \bar \varphi_\ell(s) = \bar
\varphi(s)$ if $|\bar \varphi(s)|\leq \ell$, $|\bar
\varphi_\ell(s)|\leq \ell$ for all $s$, and $ \bar \rho_\ell(s) =
\bar \rho(s)$ if $|\bar \rho(s)|\leq \ell$, $|\bar
\rho_\ell(s)|\leq \ell$ for all $s$, with $\ell$ a design
parameter.

 Let now $y_d=\mbox{col}(y, y^{(1)}, \ldots, y^{(n_y)})$, $\hat y =
 \mbox{col}(\hat y_0, \hat y_1,$ $ \ldots, \hat y_{n_y})$ and consider the change
 of variables
 $
 \hat y \; \mapsto \; e =
  D_L (y_d - \hat y)
  $
where
 $D_L= \mbox{diag}( L^{n_y}, \, L^{n_y-1},\, \ldots,\, 1 )$. In this
 coordinate setting, by denoting
 $
  {x} = \mbox{col}(\ba{cc} z & \varsigma \ea   )
 $,
 the overall closed-loop system reads as
  \beeq{\label{closedloopcompact}
 \ba{rcl}
 \dot x &=& {f}(x) + \Delta(x,e)\\
 \dot e &=& L H e + B ({q}(x) + v)
 \ea
 }
  in which $H$ is a Hurwitz matrix in observability canonical form,
  $B=\left(\ba{cccc}0 & \ldots & 0& 1  \ea \right )\tr$
  $\dot x = {f}(x)$ is a compact representation of the
  system (\ref{pla0}), (\ref{varsigmareg}), ${q}(x) = C_{n_y
+1}(z,\varsigma)$ and

 \[
  \Delta(x,e) = \left ( \ba{c} 0\\
 A(z,\bar \rho_\ell(\varsigma, y, \hat y_{1}, \ldots, \hat y_{n_y}) )
-A(z, \bar \rho(\varsigma, y, y^{(1)}, \ldots, y^{(n_y)}) )\\
 \bar \varphi_\ell(\varsigma, y, \hat y_{1}, \ldots, \hat y_{n_y}) ) - \bar \varphi(\varsigma,
y, y^{(1)}, \ldots, y^{(n_y)})  ) \ea \right )
 \]
in which, in the latter, we have left the "original" coordinates
$\hat y_i$ for notational convenience.

By definition of $\Delta$, of $\bar \rho_\ell$ and of $\bar
\varphi_\ell$, it turns out that for any bounded set ${\cal Z}'_M$
and ${\cal H} \subset \Real^{\dimxi + \dimeta}$, there exists an
$\ell^\star$ such that for any $\ell \geq \ell^\star$
\[
 \Delta(x, 0) =0 \qquad \forall \; x \in {\cal Z}'_M \times {\cal
H}\,.
\]

Furthermore, by the fact that the triplet (\ref{triplet2}) is rLER it follows that there exists a
compact set ${\cal R}' \supset \ca$ which is LES for $\dot z = A(z,\bar u(z))$ with ${\cal D}({\cal
R}') \supseteq {\cal D}(\ca)$. Let ${\cal R}'' := {\cal R} \bigcap {\cal R'}$. By item (a) above
and by going throughout the proof of Proposition \ref{theorembackstepping}, it turns out that ${\rm
graph} \left . \tau \right |_{{\cal R}''}$ is LES for $\dot x = f(x)$. Moreover, the set ${\cal
R}''$ is LES for $\dot z = A(z,\bar u(z))$ with ${\cal D}({\cal R}'') \supseteq {\cal D}(\ca)$ and,
by definition of rLER, it is possible to claim that the triplet $(A(z,\bar u(z)), L_{A(z,\bar
u(z))}^{(n_y+1)} C(z), {\cal R}'')$ is rLER. From the previous facts, from the item (b) above,
which implies that
\[
 \left . {q}(x) \right |_{{\rm graph} \left . \tau \right |_{{\cal R}''}} = C_{n_y +1}(z,
\tau(z))= L_{A(z, \bar u(z))}^{(n_y+1)}C(z)\,,
\]
 and by Proposition \ref{Propx1x2} in Appendix \ref{AppendixAuxRes}, it follows that the triplet
\[
  ({f}(x), {q}(x), {\rm graph} \left . \tau \right |_{{\cal R}''})
\]
 is LER.
Thus system (\ref{closedloopcompact}) fits in the framework of
Theorem \ref{TheoremFirst} (with graph$\left . \tau \right
|_{{\cal R}''}$ playing the role of the set $\ca$) by which  the
result follows (by using the fact that graph$\left . \tau \right
|_\ca$ is LAS for $\dot x = f(x)$ and by adapting the omega-limit
set arguments at the end of the proof of Theorem
\ref{TheoremFirst}). $\triangleleft$

\end{proof}


 \section{Sufficient conditions for exponential reproducibility}\label{SecSuff}

 Having established with Theorems \ref{TheoremFirst} and \ref{TheoremMainUCO} the interest of local exponential
reproducibility for solving the problem of (robust) output
feedback stabilization via a locally Lipschitz regulator, in this
section we present a number of results which are useful to test
when a triplet $(F,Q,\ca)$ is rLER (and thus LER) and, eventually,
to design the functions $(\varphi, \psi,\gamma)$.


   As also commented in Section \ref{SecSuff}, the first
 requirement behind the definition is the existence of a compact set ${\cal
 R} \supseteq \ca$ which is LES for $\dot z = F(z)$. In this
 respect we present a result which
 claims that the existence of a set $\cal R$ which is LES for $\dot z = F(z)$ is
 automatically guaranteed if the set $\ca$ is LAS for $\dot z=F(z)$.
 Thus, put in the context of Theorem \ref{TheoremFirst}, the first
 requirement of the definition is not restrictive at all. Details of this fact are reported in the
 following proposition whose proof can be found in \cite{MPBook}.\\[1mm]
\begin{lemma}\label{toolPropositionAp}
 Consider system
 \beeq{\label{toolsys2}
  \dot z = F(z) \qquad z \in \Real^\dimz
 }
 evolving on an invariant closed set ${\cal C}_\dimz \subset \Real^\dimz$. Let $\ca \subset {\cal C}_\dimz$
 be a compact set which is LAS with domain of attraction ${\cal
 D}(\ca) \subseteq {\cal C}_\dimz$. For any compact set ${\cal S} \subset {\cal C}_\dimz$ such that
 $\ca \subset {\rm int}\cal S$, there exists a compact set ${\cal R}$ satisfying $\ca \subseteq {\cal R}
 \subset {\cal S}$ which is LES for (\ref{toolsys2})
 with domain of attraction ${\cal D}({\cal R})\equiv{\cal D}(\ca)$.
\end{lemma}
\vspace{2mm}

  We pass now to analyze the second crucial requirement behind the definition of rLER,
  namely the existence of {\em locally Lipschitz} functions $(\varphi, \psi, \gamma)$
 and $T$ such that conditions (\ref{Timmdef}) and (\ref{condstabdefRob}) are satisfied
 for system (\ref{xisysdefRob}).
 Being the property in question related to the ability of
reproducing any signal $Q(z(t))$ generated by the system $\dot
z(t) = F(z(t))$ by taking its initial conditions in the set $\cal
R$, it is not surprising that the theory of nonlinear output
regulation, and specifically the design techniques proposed in the
related literature to construct internal models, can be
successfully used to this purpose (see \cite{CASYMI}). In
particular, in the following, we present two techniques which are
directly taken, with minor adaptations, from the literature of
output regulation.

First, we follow \cite{BI03bis} and we present a method which
draws its inspiration from high-gain design techniques of
nonlinear observers. Specifically it is possible to state the
following proposition which comes from Lemma
\ref{toolPropositionAp} and from minor adaptations\footnote{The
 adaptation consists only in proving the ISS property which is
 behind the definition of rLER.} of the main result of \cite{BI03bis} (see the quoted work for the proof).

 \begin{proposition}\label{PropositionPhase}
  Let $F: \Real^\dimz \to \Real^\dimz$ and $Q: \Real^\dimz \to
  \Real$ be given  smooth functions and $\ca \subset \Real^m$ be a given compact set which is LAS for
  $\dot z = F(z)$. Assume, in addition, that there exist a $\tilde \dimz >0$, a
compact set $\cal S$ such that $\ca \subset {\rm int } \cal S$ and
a locally Lipschitz function $f: \Real^{\tilde \dimz} \to \Real$
such that the following differential equation holds
 \beeq{\label{immobser}
 L_{F(z)}^{\tilde \dimz} Q(z) = f(\,Q(z), \,L_{F(z)}Q(z), \ldots,  \,L_{F(z)}^{\tilde \dimz -1}Q(z)\,)
 \qquad \forall \, z \in {\cal S}\,.
 }
  Then the triplet $(F,Q,\ca)$ is rLER. In particular $(\varphi,\psi,\gamma)$ can be taken as
  the functions $\varphi: \Real^{\tilde \dimz} \to \Real^{\tilde \dimz}$, $\psi:
\Real^{\tilde \dimz} \to \Real^{\tilde \dimz}$, $ \gamma:
\Real^{\tilde \dimz} \to \Real$ defined as
\[
 \varphi( \xi) = \left(  \ba{c} \xi_2 + \lambda_0 L  \xi_1\\
  \xi_3 + \lambda_1 L^2 \xi_1\\
 \vdots\\
 \xi_{ \tilde \dimz} + \lambda_{ \tilde \dimz -2} L^{ \tilde \dimz -1}
  \xi_1\\
 f_c({\xi}_1,\, {\xi}_2,\, \ldots, \, {\xi}_{\tilde \dimz}) + \lambda_{\dimz -1} L^{ \tilde \dimz}
 \xi_1\\
 \ea
\right )\,, \qquad
 \psi( \xi) = \left(  \ba{c} -\lambda_0 L\\
-\lambda_1 L^2\\ \vdots\\
 -\lambda_{\tilde \dimz -2} L^{ \tilde \dimz -1}\\
 -\lambda_{\tilde \dimz -1} L^{\tilde \dimz}
 \ea \right )
\]
and $\gamma(\xi) =\xi_1$, where $L$ is a positive design parameter
to be taken sufficiently large, $\lambda_i$, $i=0,\ldots, \tilde
\dimz-1$, are such the polynomial $s^{\tilde \dimz} +
\lambda_{\tilde \dimz-1} s^{\tilde \dimz-1} + \ldots + \lambda_1 s
+ \lambda_0$ is Hurwitz, and $f_c(\cdot)$ is any bounded function
 such that  $f_c \circ \tau(z) = f(z)$ for all $z \in \cal S$
 where $\tau: \Real^{\dimz} \to \Real^{\tilde \dimz}$ is defined
 as
  \[\tau(z) = \left (\ba{cccc}Q(z) & L_{F(z)}Q(z) & \cdots & L_{F(z)}^{\tilde \dimz -1} Q(z) \ea  \right
   )\tr\,.
   \]
 \end{proposition}

 \begin{remark} \label{remarkNoDiffObs}
 It is well-known (see, for instance, \cite{GKbook}) that a sufficient condition for a pair $(F,Q)$
 to satisfy property (\ref{immobser}) locally with respect to a point $z_0$ is that its observability
 distribution at $z_0$ (see \cite{HerKre})
 \beeq{\label{ObsvDistr}
 \Omega_\dimz (z) = \sum_{k=0}^{\dimz-1} {\rm span}{\partial \over \partial
 z}L_{F(z)}^k Q(z)
 }
 has dimension $\dimz$ at $z=z_0$, namely if the system
 $\dot z = F(z)$ with output $y_z=Q(z)$ satisfies the
 {\em observability rank condition} (by using the terminology of \cite{HerKre}) at $z_0$.
 Such a condition represents an observability
 condition for the system $\dot z = F(z)$ with output $y_z=Q(z)$
 which, however, is far to be necessary to fulfill the property of rLER. In this respect
 it must be stressed again that the
 property of local exponential reproducibility does not involve
 any state observability property of system $\dot z = F(z)$ with output $y_z=Q(z)$
 but rather a property of output reproducibility. $\triangleleft$
\end{remark}

 Clearly the high-gain technique to design output observer behind Proposition
 \ref{PropositionPhase} is not the only tool which can be used to
 design the functions $(\varphi,\psi,\gamma)$ for a triplet which is rLER.
 In order to enrich the available tools, we present now a result
 motivated by the theory of state observers pioneered in \cite{KaKr} (and
 developed in \cite{Andrieu}, \cite{Krener}) which, in turn, has
 inspired the technique to design internal models developed in \cite{SICON}
 in the context of nonlinear output regulation (see also \cite{MPTAC}).
 In this respect it is interesting to observe that if instead of asking $(\varphi, \psi, \gamma)$ to
 be locally Lipschitz these functions were required to be only
 continuous, the
 theoretical tools presented in \cite{SICON} are sufficient to prove that any smooth
 triplet $(F,Q,\ca)$ is rLER if $\ca$ is LAS for $\dot z=F(z)$. In
 particular,  if $(H,G) \in \Real^{\dimxi \times \dimxi}
 \times \Real^{\dimxi \times 1}$, $\dimxi >0$, is an arbitrary controllable pair with $H$ a Hurwitz matrix,
 and $\cal R$ is the set which is LES for $\dot z = F(z)$ (whose existence is guaranteed
 by Lemma \ref{toolPropositionAp} since $\ca$ is LAS for these dynamics), then it turns out that if
 the functions $\varphi$ and $\psi$ in (\ref{xisysdefRob}) are chosen as
\[
 \varphi(\xi) = H \xi \qquad\mbox{and}\qquad \psi(\xi) = G
\]
 then property (\ref{condstabdefRob}) holds true  with (see Propositions 1 in \cite{SICON})
 \beeq{ \label{Tdefinition}
 T(z) =  \int_{-\infty}^0 e^{-Ht} G Q(\Phi_F(t,z)) dt\,.
 }

 Furthermore, if $\dimxi$ is chosen so that $\dimxi \geq 2 \dimz + 2$ and the matrix $H$ is taken so that
 $\sigma(H) \in \{\zeta \in \Compl: \Re(\zeta ) < -\ell\} \setminus \cal S$ where ${\cal S} \in \Compl$ is a set of
 zero Lebesgue measure and $\ell$ is a sufficiently large positive number, then there always
 exists a class-$\cal K$ function $\rho(\cdot)$ such that the
 following partial injectivity condition
 \beeq{\label{partialinj}
  |T(z_1)-T(z_2)| \leq \rho(|Q(z_1)-Q(z_2)|) \qquad \forall \,
  z_1,z_2 \in {\cal R}
 }
 holds (see Proposition 2 in \cite{SICON}) and, in turn, the
 latter guarantees the existence of a {\em continuous} function $\gamma:
 \Real^\dimxi \to \Real$ such that also property (\ref{Timmdef}) holds true (see Proposition 3 in \cite{SICON}).
 As shown in \cite{MPTAC} (see Proposition 4 there), a possible expression of such a $\gamma(\cdot)$
 is given by
 \beeq{\label{gammaMCS}
 \gamma(\xi) = \inf_{z \in {\cal R}} \{ - Q(z)+ \rho (|\xi -
 T(z)|)\}\,.
 }

  The previous arguments, however, are not conclusive if system (\ref{xisys}) is required to be
  locally Lipschitz. In this respect an extra condition (see the forthcoming (\ref{Omegadef}),
  (\ref{rankcond})) is needed to guarantee the existence of a
 locally Lipschitz $\gamma$ as precisely proved in \cite{AddendumSICON}. The main results in this direction
 are presented
 in the next lemma (whose proof is a minor adaptation of the main result of \cite{AddendumSICON}).\\[1mm]
 \begin{proposition} \label{lemmarank}
 Let $\cal O$ be an open bounded set which is
 backward invariant for $\dot z = F(z)$ with $F:\Real^\dimz \to \Real^\dimz$ a smooth
 function. Let $\ca \subset \cal O$ be a compact set which is LAS for $\dot z = F(z)$ and let
 $Q: \Real^\dimz \to \Real$ be a $C^\infty$
 function.
Let $\Omega (z)$ be the distribution defined as
 \beeq{\label{Omegadef}
  \Omega(z) :=  \sum_{k=0}^\infty {\rm span}
  {\partial \over \partial z} L^k_{F(z)} Q(
  z)\,.
 }
 If there exists a constant $c \leq \dimz$ such that
 \beeq{\label{rankcond}
 {\rm dim} \Omega (z) = c \qquad \forall \, z \in {\cal O}
 }
 then for any compact set ${\cal R} \subset \cal O$, with ${\cal R} \supseteq
 \ca$, the function $\rho$ in (\ref{partialinj}) can be taken linear and
 the function $\gamma$ in (\ref{gammaMCS}) is Locally Lipschitz.
 As a consequence the triplet $(F, Q, {\cal A})$ is rLER.
 \end{proposition}
 \vspace{2mm}

 \begin{remark}
 By also bearing in mind Remark \ref{remarkNoDiffObs}, it is worth
 noting how condition (\ref{rankcond}) {\em does not} require the
 observability distribution (\ref{ObsvDistr})to be full rank, namely
 the system $\dot z = F(z)$ with output $y_z = Q(z)$
 to be observable in any sense. Rather, condition (\ref{rankcond}) can be regarded as a
 regularity condition of the observable part of system $\dot z = F(z)$ with output $y_z =
 Q(z)$. $\triangleleft$
 \end{remark}

\begin{remark}
 By going throughout the technical details in \cite{AddendumSICON}, it is possible to observe that the
 requirement about the existence of an open bounded set $\cal O$ which is {\em backward invariant} for $\dot z = F(z)$
 is {\em uniquely} motivated by the need of having the function $T(z)$ in (\ref{Tdefinition}) well-defined
 and $C^2$. In
 this respect the requirement
 about the existence of a bounded invariant set $\cal O$ can be substituted  by the requirement that
 the function $T(z)$ in (\ref{Tdefinition}) is well-defined and $C^2$ for all $z \in {\rm int} \cal R$
 for a proper choice of the matrix $H$. In this case, by the details in \cite{AddendumSICON},
 it turns out that the rank condition (\ref{rankcond}) must be substituted by
 \[
 {\rm dim} \Omega (z) = c \qquad \forall \; z \in
   \{\Phi_F(t,\varsigma)\quad : \quad t\leq 0\,, \; \varsigma \in {\cal
   R}\}\,.\,\triangleleft
 \]
 \end{remark}

\begin{remark}\label{remarkclipping}
 The requirement in Lemma \ref{lemmarank} about the existence of a
 bounded set $\cal O$ which is backward invariant for $\dot z = F(z)$, may
 be practically overtaken by properly "clipping" the function $F(z)$ outside the set $\ca$.
 As a matter of fact, being the
 property of rLER related to the ability of reproducing the
 signals $Q(z(t))$ generated by the system $\dot z(t) = F(z(t))$
 by taking its initial conditions in a neighborhood of $\ca$, it turns out that
 any triplet $(F_c,Q_c,\ca)$, with $F_c : \Real^\dimz \to \Real^\dimz$, $Q_c:\Real^\dimz \to \Real$
 functions which agree with $F$ and $Q$ on some compact $\cal S$ containing $\ca$ in its interior,
 can be used in place of $(F,Q,\ca)$ to check whether the latter is rLER and eventually
 to design the functions $(\varphi,\psi,\gamma)$. In this respect,
 the presence of a bounded backward invariant set $\cal O$ may be forced
 by properly clipping to zero the vector field $F(z)$ outside
 $\ca$. Alternatively, by bearing the previous remark, the function $T$ be forced
 to be well-defined and $C^2$ by properly clipping the function $Q(z)$ outside
 $\ca$. For reason of space we omit the technical details to rigorously prove the previous
 intuition and we refer the reader to the example in Section \ref{SecExample} for an illustrative
 example.
 \end{remark}

\section{Example}\label{SecExample}

Consider the system
 \beeq{\label{sysExa}
 \ba{rcl}
 \dot w &=& 0\\
 \dot x &=& -x^3 + w + y\\
 \dot y &=& x + u
 \ea
 }
 with control input $u \in \Real$, measured output $y\in \Real$ in
 which $w$ is a constant signal taking value in the interval $W:=[\underline{w},
 \overline{w}]$. By defining ${\cal C}_1 :=W \times \Real$ and
\[
  \ca:= \{(w,x) \in {\cal C}_1\; : \; x = \sqrt[3]{w} \}
\]
 we address  the problem of stabilizing the
 set $\ca \times \{0\}$, which is invariant for the previous
 system with $u=0$, by means of a $y$-feedback. The set $\ca$ is LAS for the
 zero dynamics of (\ref{sysExa}) with domain of attraction ${\cal D}(\ca) =  {\cal C}_1$
 and, by defining  $u= - \kappa y + v$, the previous problem fits in the
 framework of Section \ref{SecFramework}. Note, in particular,
 that by only increasing the value of $k$ while setting $v=0$ the
 desired asymptotic stabilization objective can not be met. As a
 matter of fact, different equilibria characterize the system according to the value of $w$.
 For $w=0$, the $y$-component of the system has three equilibria
 given by $(0,\sqrt{1/\kappa^3}, -\sqrt{1/\kappa^3})$ which, for $w \neq
 0$, collapse in one equilibrium, solution of $w+y = \kappa^3y^3$,
 which tends to $0$ as $k$ tends to $\infty$. So, with $v=0$, only
 {\em practical} stabilization of the set $\ca \times \{0\}$ in the parameter $k$ can be
 achieved. In order to apply Theorem \ref{TheoremFirst} and to obtain
 asymptotic stabilization of the set $\ca \times \{0\}$ by means
 of dynamic feedback, we let $z:= \mbox{col}(w,x)$,
  \[
  F(z) := \left ( \ba{c} 0\\ -x^3 + w \ea \right ) \qquad Q(z) := x
 \]
 and we check the exponential reproducibility of the
 triplet $(F(z),Q(z),\ca)$. To this purpose, let
 $\cal S$ be a compact set of the form
 \[
 {\cal S} = \{(w, x) \in {\cal C}_1 \; : \; \sqrt[3]{\underline{w}}- {1 \over 2} \leq x \leq \sqrt[3]{\overline{w}} +
 {1 \over 2} \}
 \]
 so that $\ca \subset {\rm int} \cal S$, and note that, by Lemma
 \ref{toolPropositionAp}, there exists a set ${\cal R}$, $\ca \subseteq {\cal R} \subset \mbox{int} \cal
 S$, which is LER with ${\cal D}({\cal R}) = {\cal D}(\ca)$.
 Both Proposition \ref{PropositionPhase} and \ref{lemmarank} can
 be used to prove that the triplet in question is rLER and, indeed,
 to design the functions $(\varphi,\psi,\gamma)$.

  By following Proposition \ref{PropositionPhase}, it is easy
  to check that
  \[
 L_{F(z)}^2Q (z) = -3 \, Q^2(z) \, L_{F(z)}Q (z) \qquad \forall \,
 z \in {\cal C}_1\,,
  \]
  namely condition (\ref{immobser}) holds (with $\tilde m =2$), and thus the triplet $(F(z),Q(z),\ca)$ is rLER. According
  to the proposition, the
  functions $(\varphi,\psi,\gamma)$ can be designed as
  \beeq{\label{Ex1Triplet1}
   \varphi(\xi) = \left ( \ba{c} \xi_2 + \lambda_0 L \xi_1\\
   f_c(\xi_1,\xi_2) + \lambda_1 L^2 \xi_1
  \ea
     \right )
\qquad
 \psi(\xi) = \left(  \ba{c} -\lambda_0 L\\ - \lambda_1 L^2 \ea
 \right )
 \qquad \gamma(\xi)= \xi_1
  }
 where $\lambda_0$ and $\lambda_1$ are such that $s^2 + \lambda_1 s +
 \lambda_0$ is an Hurwitz polynomial, $L$ is a sufficiently large design parameter and $f_c(\cdot)$ is any smooth
 bounded function such that
 \[
 f_c(x, \,-x^3+w) = -3 x^2 (-x^3+w)   \qquad \forall \, (x,w) \in {\cal
 S}\,.
  \]

 The functions $(\varphi,\psi,\gamma)$ can be designed according
 to Proposition \ref{lemmarank} as well. However note that this proposition
 cannot be applied as such  due to the  absence of a bounded invariant set
 containing ${\ca}$ (indeed finite escape time occur in backward time for initial conditions
 outside $\ca$). To overtake this obstacle and by bearing in mind
 Remark \ref{remarkclipping}, pick a smooth function $a:\Real \to \Real_{\geq 0}$ such that
$a(s) =1$ for all $\sqrt[3]{\underline{w}}- {1 \over 2} \leq s
\leq \sqrt[3]{\overline{w}} + {1 \over 2}$ and $a(s) = 0$ for all
$s \leq \sqrt[3]{\underline{w}}- 1$ and $s \geq
\sqrt[3]{\underline{w}}+ 1$,  and consider the system $\dot {z} =
a(x) F(z)$. For this system the bounded set $\cal O$ defined as
 \[
 {\cal O} = \{ z \in {\cal C}_1 \; : \; \sqrt[3]{\underline{w}}- 1 \leq x \leq
 \sqrt[3]{\overline{w}} + 1 \}\,,
 \]
 which is open with respect to the subset topology induced by ${\cal
 C}_1$, is invariant. Thus, Lemma
 \ref{lemmarank} can be applied to the triplet $(a(x)F(z), Q(z), {\ca})$.
 In this specific case
 \[
 \Omega(z) = {\rm span} \left( \ba{cc} 0 & 1\\
 a(x) & \ast \ea \right )
 \]
 where $\ast$ is a junk term, from which it follows that
 \[
 {\rm dim}\Omega(z) \equiv 2 \qquad \forall \; z\in {\cal O}
 \]
 which implies, by Proposition \ref{lemmarank}, that the triplet $(a(x) F(z), Q(z), {\ca})$
 is rLER. But, as the functions $F(z)$ and $a(x)F(z)$ agree on $\cal
 S$, the fact that the triplet $(a(x) F(z), Q(z),
 {\ca})$ is rLER can be shown to imply that also the triplet $(a(x) F(z), Q(z),
 {\ca})$ is such. Thus, according to Proposition
 \ref{lemmarank}, the functions $(\varphi,\psi,\gamma)$ can be
 also designed as
 \beeq{\label{Ex1Triplet2}
 \varphi(\xi) = H \xi\,, \qquad \psi(\xi)=G\,, \qquad \gamma(\xi)
 = \inf_{z \in {\cal R}} \{ - Q(z)+ \rho \cdot |\xi -
 T(z)| \}
 }
where $(H,G) \in \Real^{5\times 5} \times \Real^{5 \times 1}$ is
an arbitrary controllable pair with the matrix $H$ such that
$\sigma(H) \in \{\zeta \in \Compl: \Re(\zeta ) < - \ell\}
\setminus \cal S$ where ${\cal S} \in \Compl$ is a set of
 zero Lebesgue measure and $\ell$ is a sufficiently large positive
 number, the function $T$ defined as in (\ref{Tdefinition}) and
 $\rho$ a sufficiently large positive number.\\

  As a result the problem of output feedback stabilization of the
  set $\ca \times \{0\}$ can be achieved, by bearing in mind Theorem \ref{TheoremFirst}  and the subsequent
  remark, by the following dynamic controller
 \[
 \ba{rcl}
  \dot \xi &=& \varphi(\xi) - \psi(\xi)[\gamma(\xi) - \kappa y]\\
  u &=& - \kappa y + \gamma(\xi)
 \ea
 \]
 with the functions $(\varphi,\psi,\gamma)$ designed as in
 (\ref{Ex1Triplet1}) or (\ref{Ex1Triplet2}) and $\kappa$ a
 sufficiently large number.

\begin{appendix}


 \section{Proof of Theorem \ref{TheoremFirst}} \label{ProofTheoremFirst}

 By the definition of LER of the triplet $({f}(x,0),{q}(x,0),
 {\ca})$  there  exists a set ${\cal R}\supseteq
 \ca$ which is LES for $\dot x = f(x,0)$ and , for any compact set ${\cal X}_1 \subset {\cal D}({\cal
 R})$,
 there exist an integer $\dimeta$, locally  Lipschitz functions $\varphi:\Real^\dimeta \to
 \Real^\dimeta$, $\psi: \Real^\dimeta \to \Real^\dimeta$, $\gamma: \Real^\dimeta \to
 \Real$ and a smooth function $T: \Real^\dimx \to \Real^\dimeta$
 such that
  \beeq{\label{TimmdefProof}
    \ba{rcl}
       q(x,0) + \gamma(T(x)) =0
    \ea \qquad \forall \; x \in {\cal R}
   }
   and for all $\xi_0 \in \Real^\dimxi$ and
   $x_0 \in {\cal X}_1$  the solution $(\xi(t), x(t))$ of

       \beeq{\label{xisysdefProof}
        \ba{rcll}
         \dot x &=& f(x,0) & x(0) = x_0\\
            \dot \xi &=& \varphi(\xi) +
            \psi(\xi)\, q(x,0) & \xi(0) = \xi_0
   \ea
     }
     satisfies
     \beeq{\label{condstabdefProof}
        |(\xi(t),x(t))|_{{\rm graph }\left. T \right |_{\cal R}}
        \leq \beta(t, |(\xi_0, x_0)|_{{\rm graph }\left . T \right |_{\cal R}
        })
     }
    where $\beta(\cdot,\cdot)$ is a locally exponentially class-$\cal
    KL$ function. As $\ca$ is LAS($\cal X$) and ${\cal R}\supseteq \ca$,
    we have ${\cal D}(\cal R) = {\cal D}(\ca)$ and thus the
    previous properties hold, in particular, with ${\cal X}_1 = \cal
    X$. Furthermore, in case $\ca$ is LES for $\dot x = f(x,0)$, it
    is possible to show\footnote{Internal remark: to be checked.} that (\ref{TimmdefProof})
    and (\ref{condstabdefProof}) hold also with $\cal R$ replaced by
    $\ca$ possibly with a different class-$\cal KL$ function $\beta(\cdot,
    \cdot)$.

 Assume, without loss of generality (as $(A,B,C)$ has relative degree $r$ and $(A,C)$ is observable),
 that the pair $(A,C)$ is in the canonical observability
 form and that $B=(0,\ldots, 0,1)\tr$, and
 choose, as candidate controller, the system
 \beeq{\label{regth1}
  \ba{rcl}
 \dot \eta &=& \varphi(\eta)  - \psi(\eta)\gamma(\eta) - \psi(\eta) \, \kappa B\tr A y\\
  v &=&\gamma(\eta)
  \ea
}
 which, by the structure of $A$ and $B$, is of the form
(\ref{etasys}) since $B^T A y = a_r y_m$ for some real number
$a_r$.

Consider now the change of variables
 \beeq{\label{chi-etachange}
  \eta \; \to \; \chi :=\phi_{\psi}(B\tr y, \eta)\,.
 }
  Note that such a change of variables is
well-defined for all $y$ and $\eta$ as $\psi$ is complete.

Since
\[
 {\partial \phi_{\psi}(t',\eta) \over \partial t'} -
 {\partial \phi_{\psi}(t',\eta) \over \partial \eta} \psi(\eta)
 \equiv 0
\]
and using the fact that
 \[
 \left . {\partial \phi_{\psi}(t',\eta) \over \partial t'}
 \right |_{t'=0} = \psi(\eta)
 \qquad \quad
 \left . {\partial \phi_{\psi}(t',\eta) \over \partial \eta}
 \right |_{t'=0} = I\,,
 \]
  it turns out that the closed-loop dynamics (\ref{plantstab}), (\ref{regth1}) in the new coordinates  can be
described as the feedback interconnection of a system of the form
 \beeq{\label{upperth1}
 \ba{rcl}
 \dot x &=& {f}(x,0) + \tilde f(x,y)\\
 \dot \chi &=& \varphi(\chi) + \psi(\chi){q}(x,0) + \tilde \ell_1(x,\chi,y)\\
 \ea
 }
 and a system of the form
 \beeq{\label{lowerth1}
 \dot y = \kappa A y + B({q}(x,0) + \gamma(\chi)) +
 \tilde \ell_2(x,\chi,y) }

  in which $\tilde f(x,y)$, $\tilde \ell_1(x,\chi,y)$ and
$\tilde \ell_2(x,\chi,y)$ are {\em locally Lipschitz} functions
satisfying $\tilde f(x,0)=0$, $\tilde \ell_1(x,\chi,0)=0$ and
$\tilde \ell_2(x,\chi,0)=0$ for all $x \in \cal C$ and $\chi \in
\Real^\dimeta$ and with $\tilde \ell_1$ and $\tilde \ell_2$
possibly dependent on ${\cal X}$.

 Let $\cal N$ be an arbitrary compact set of
$\Real^\dimeta$ and denote by $\Xi$ the image of $\cal N$ under
the function (\ref{chi-etachange}) (note that $\cal N$ may depend
on ${\cal X}$). Since system (\ref{upperth1}) with $y=0$ is
nothing but (\ref{xisysdefProof}), it turns out that ${\rm graph}
\left . T \right |_{\cal R}$ is LES(${\cal X} \times \Xi$) for
system (\ref{upperth1}) with $y=0$.
 Furthermore, by (\ref{TimmdefProof}), the term ${q}(x,0) + \gamma(\chi)$
 in (\ref{lowerth1}) is identically zero for $(x,\chi) \in {\rm graph} \left . T \right |_{\cal R}$.
 From these facts and the results in \cite{TP}, \cite{Atassi}, it follows that for any
 compact set ${\cal Y} \in \Real^\dimy$ there exists a
 $\kappa^\star>0$ such that for all $\kappa \geq \kappa^\star$ the set
 ${\rm graph}\left . T\right |_{\cal R} \times \{0\}$ is LES(${\cal X} \times
\Xi \times \cal Y$) for (\ref{upperth1}), (\ref{lowerth1}). By
taking $\tau = \left . T \right |_{\cal R}$ the previous result
proves the first part of the theorem,  namely that ${\rm graph}\,
\tau \times \{ 0\}$ is LES(${\cal X} \times {\cal Y} \times \cal
N$) for the closed-loop system (\ref{plantstab}), (\ref{regth1}).

We prove now the second claim of the theorem, namely that ${\rm
graph}\left . \tau \right|_\ca \times \{0\}$ is LAS(${\cal X}
\times {\cal Y} \times \cal N$).
  Let $\kappa \geq \kappa^\star$ be fixed and note that, as
 ${\rm graph}\tau  \times \{0\}$
 attracts uniformly the closed-loop trajectories leaving ${\cal X} \times {\cal
  Y} \times \cal N$, Proposition \ref{toolPropomega} yields that
 \[
 \omega( {\cal X} \times {\cal Y} \times {\cal N}) =
 \omega({\rm graph}\tau \times \{0\}) \subseteq {\rm graph}\tau \times \{0\}
 \]
 in which $\omega(S)$ denotes the omega limit set of the set $S$ associated
 to the closed-loop system.
  We prove now that if $(x,y,\chi) \in \omega( {\rm graph}\tau \times \{0\})$
 then necessarily $x \in \omega({\cal R})$ in which $\omega({\cal R})$
 denotes the omega limit set of the set ${\cal R}$ associated to the system
 $\dot x = f(x,0)$.

Indeed, consider
  a sequence $\{x_n, y_n, \chi_n\}$ with  $(x_n,\chi_n)\in
{\rm graph}\tau$ and so in particular $x_n \in {\cal R}$, and $y_n
\equiv 0$,
  and a divergent sequence $\{t_n\}$, such that, the following holds
  \beeq{\label{ccth1}
   |x(t_n,x_n) - \bar x| \to 0 \,  ,
  }
where $x(t, x_n)$ and $\chi((x_n,\chi_n),t)$ denotes the solution
of
   \beeq{\label{triangular}
  \ba{rcl}
  \dot x &=& {f}(x, 0)\\
  \dot \chi &=& \varphi(\chi) + \psi(\chi)q(x,0)
  \ea
  }
   with initial conditions $(x_n,\chi_n)$.
$x_n$ being in ${\cal R}$, this implies $\bar x \in \omega ({\cal
R})$.
  Now, considering the system given by the first dynamics in (\ref{triangular})
  and using the  fact that $\ca \subseteq \cal R$ uniform attracts the trajectories
  of this system leaving $\cal X$, Proposition \ref{toolPropomega}
  in Appendix
  yields that  $\omega({\cal X}) = \omega({\cal R}) = \omega(\ca)
  \subseteq \ca$. By this and the previous arguments we conclude that
  the $x$ components of the closed-loop trajectories are uniformly
  attracted by $\omega(\ca) \subseteq \ca$. From this the result follows by
  standard arguments. $\triangleleft$


\section{Auxiliary results} \label{AppendixAuxRes}

\begin{proposition} \label{toolPropomega}
Let
 \beeq{\label{Fsysapp}
 \dot z = F(z)
 }
 be a given smooth system and let $S$ be a compact set which is forward invariant for (\ref{Fsysapp})
 and which uniformly (in the initial condition) attracts the trajectories of (\ref{Fsysapp}) originating in a compact set $D \supset S$.
 Then $\omega(D) = \omega(S) \subseteq S$.
\end{proposition}
{\em Proof}
  First of all note that $\omega(D)$ and $\omega(S)$ exist and that, by
  definition, $\omega(S) \subseteq \omega(D)$. Furthermore $\omega(S) \subseteq S$
  as $S$ is forward invariant for (\ref{Fsysapp}).
  To prove that $\omega(D) = \omega(S)$ suppose that it is not, namely that
  there exist a $\bar z \in \omega(D)$ and an $\epsilon>0$
  such that $|\bar z|_S \geq \epsilon$. As $S$ uniformly attracts the trajectories of (\ref{Fsysapp})
  originating from $D$, there exists a $t_{\epsilon/2}>0$ such that $|z(t, z_0)|_S
  \leq \epsilon/2$ for all $z_0 \in D$ and for all $t \geq t_{\epsilon/2}$.
 Moreover, by definition of $\omega(D)$,
  there exist sequences $\{z_n\}_0^\infty$ and $\{t_n\}_0^\infty$, with
  $z_n \in  D$ and $\lim_{n \rightarrow \infty}t_n = \infty$, such that
  $\lim_{n \rightarrow \infty}z(t_n, z_n) =\bar z$. This, in particular,
  implies that for any $\nu>0$ there exists a $n_\nu>0$ such that $|z(t_n,z_n)
  -\bar z| \leq \nu$ for all $n \geq n_\nu$. But, by taking $\nu = \min\{
  \epsilon/2,\nu_1\}$ with $\nu_1$ such that $t_n \geq t_{\epsilon/2}$ for all
  $n \geq n_{\nu_1}$,  this contradicts that $S$ uniformly attracts the trajectories
  of the system originating from $D$. $\triangleleft$


\begin{proposition} \label{Propx1x2}
 Consider a system of the form
 \beeq{\label{x1x2sys}
  \ba{rcl}
  \dot x_1 &=& f_1(x_1,x_2) \qquad x_1 \in \Real^{\dimx_1}\\
  \dot x_2 &=& f_2(x_1,x_2) \qquad x_2 \in \Real^{\dimx_2}
  \ea
 }
 and assume that there exist a compact set $\ca \subset
 \Real^{\dimx_1}$ and a smooth function $\tau : \Real^{\dimx_1} \to
 \Real^{\dimx_2}$  such that the set
 \[
  {\rm graph} \left . \tau \right |_\ca = \{(x_1,x_2)\in \ca \times \Real^{\dimx_2} \quad : \quad x_2 =
  \tau(x_1)\}
  \]
  is LES for (\ref{x1x2sys}) and the set $\ca$ is LES for the system $\dot x_1 = f_1(x_1,\tau(x_1))$.
  Let $q : \Real^{\dimx_1} \times \Real^{\dimx_2} \to \Real$ be a smooth
  function. If the triplet $(f_1(x_1,\tau(x_1)),$ $ q(x_1, \tau(x_1)),
  \ca)$ is rLER then the triplet $({\rm col}(f_1(x_1,x_2),\, f_2(x_1,x_2)), q(x_1,x_2),
  {\rm graph} \left . \tau \right |_\ca)$ is LER.
\end{proposition}
 \begin{proof}
 Let $\bar f_1(x_1) =  f_1(x_1,\tau(x_1))$ and $\bar q(x_1) =
 q(x_1,\tau(x_1))$. Since $\ca$ is LES for $\dot x_1 = \bar f_1(x_1)$ and the triplet
$(\bar f_1(x_1),$ $ \bar q(x_1), \ca)$ is rLER,   for any
 compact set $\bar X_1 \subset {\cal D}({\ca})$, there
 exist an integer $\bar \dimxi$ and locally Lipschitz functions $\bar \varphi:
 \Real^{\bar \dimxi} \to \Real^{\bar \dimxi}$, $\bar \psi: \Real^{\bar \dimxi}
 \to \Real$, $\bar \gamma: \Real^{\bar \dimxi} \to \Real$ and $\bar T: \Real^{\dimx_1} \to \Real^{\bar \dimxi}$
 such that, for all $\bar x_{10} \in \bar X_1$ and $\bar \xi_0 \in \Real^{\bar
 \dimxi}$ and for all locally essentially bounded $v(t)$, the solution $(\bar x_1(t), \bar \xi(t))$ of the system
 \[
 \ba{rcll}
 \dot {\bar x}_1 &=& \bar f_1(\bar x_1) & \qquad \bar x_1(0) =
 \bar x_{10}\\
 \dot {\bar \xi} &=& \bar \varphi(\bar \xi) + \bar \psi(\bar
 \xi)[\bar q(\bar x_1) + v] & \qquad \bar \xi(0) = \bar \xi_0
 \ea
 \]
 satisfies
  \beeq{\label{estim1}
        |(\bar x_1(t),\bar \xi(t))|_{{\rm graph}\left. \bar T \right |_{{\ca}}}
        \leq \beta_1(t, |(\bar x_{10}, \bar \xi_0)|_{{\rm graph}\left . \bar T \right |_{{\ca}}
        })+ \ell(\sup_{\tau \leq t} |v(\tau)|)
     }
where $\beta_1(\cdot,\cdot)$ and $\ell(\cdot)$ are respectively  a
locally exponential class-$\cal KL$ and a class-$\cal K$
functions, and
\[
 \bar \gamma(\bar T(\bar x_1)) + \bar q(\bar x_1) =0 \qquad
 \forall \, \bar x_1 \in {\ca }\,.
\]
Furthermore, by the assumption that ${\rm graph} \left . \tau
\right |_\ca$ is LES for (\ref{x1x2sys}), for any $(x_{10},
x_{20}) \in {\cal D}({\rm graph} \left . \tau \right |_\ca)$ the
solution $(x_1(t), x_2(t))$ of (\ref{x1x2sys}) with initial
condition $(x_1(0), x_2(0)) = (x_{10}, x_{20})$ satisfies
 \beeq{\label{ineq11}
 |(x_1(t),x_2(t))|_{{\rm graph} \left . \tau \right |_{\ca}} \leq
 \beta_2(t, |(x_{10},x_{20})|_{{\rm graph} \left . \tau \right |_{\ca}})
 }
where $\beta_2(\cdot,\cdot)$ is a locally exponential class-$\cal
KL$ function.

Now pick a compact set $\bar X_1 \subset {\cal D}({\ca})$ and the
functions $(\bar \varphi(\cdot), \bar \psi(\cdot), \bar
\gamma(\cdot))$ accordingly, and denote by
$(x_1(t),x_2(t),\xi(t))$ the solution of the system

 \beeq{\label{eq2}
 \ba{rcll}
    \dot x_1 &=& f_1(x_1,x_2) &\qquad x_1 \in \Real^{\dimx_1}\\
  \dot x_2 &=& f_2(x_1,x_2) & \qquad x_2 \in \Real^{\dimx_2}\\
 \dot {\xi} &=& \bar \varphi(\xi) + \bar \psi(
 \xi) q(x_1,x_2) & \qquad \xi \in \Real^{\bar \dimxi}
 \ea
 }
 with initial conditions $(x_{10}, x_{20}, \xi_0) \in  \Real^{\dimx_1} \times \Real^{\dimx_2} \times
 \Real^{\bar \dimxi}$ at time $t=0$. Let $X \subset \Real^{\dimx_1} \times \Real^{\dimx_2}$
 be an arbitrary compact set such that $X \subset {\cal D}({\rm graph} \left . \tau \right
 |_{\ca})$, let ${\cal R} := {\rm graph} \left . \tau \right
 |_{\ca}$ and let $T:\Real^{\dimx_1} \times \Real^{\dimx_2} \to \Real^{\bar
 \dimxi}$ be the locally Lipschitz function defined as $T(x_1,x_2) = \bar T(x_1)$.
  We shall prove in the following that for any initial condition  $(x_{10}, x_{20}, \xi_0)
 \in X \times \Real^{\bar \dimxi}$ the trajectory
 $(x_1(t),x_2(t),\xi(t))$ of (\ref{eq2}) satisfies

 \beeq{\label{beta3}
 |(x_1(t),x_2(t),\xi(t))|_{{\rm graph}\left . T \right |_{\cal R}}
 \leq \beta_3(t, |(x_{10}, x_{20}, \xi_0)|_{{\rm graph}
 \left . T \right |_{\cal R}})
 }
 where $\beta_3(\cdot,\cdot)$ is a locally exponential class-$\cal
KL$ function and
\[
 {\rm graph}\left . T \right |_{\cal R} = \{((x_1,x_2),\xi) \in {\rm graph} \left . \tau \right
 |_{\ca} \times \Real^{\bar \dimxi} \; : \; \xi = T(x_1,x_2) \}\,.
\]
 To this purpose, pick any $\bar x_{10} \in \ca \subset \bar
 X_1$ and note that $(x_1(t),x_2(t),\xi(t))$  satisfies

 \beeq{\label{eq3}
 \ba{rcll}
    \dot x_1(t) &=& f_1(x_1(t),x_2(t)) &\qquad x_1(0)=x_{10}\\
  \dot x_2(t) &=& f_2(x_1(t),x_2(t)) & \qquad x_2(0) = x_{20}\\
 \dot {\xi}(t) &=& \bar \varphi(\xi(t)) + \bar \psi(
 \xi(t)) [\bar q(\bar x_1(t)) + v(t)] & \qquad \xi(0)=\xi_0
 \ea
 }
 where
 \[
 v(t) = q(x_1(t),x_2(t)) - \bar q(\bar x_1(t))
 \]
 and $\bar x_1(t) = \Phi_{\bar f_1}(t,\bar x_{10})$. Let $x_1^\star \in \ca$ be such that $(x_1^\star,\tau(x_1^\star))$ is
 the projection of $(x_1,x_2)$ on ${\rm graph} \left . \tau \right
 |_\ca$. Since $\bar x_1(t) \in
 \ca$ and $x_1^\star(t) \in \ca$ for all $t \geq 0$, $X$ is compact and (\ref{ineq11}) holds,
 and $\bar q$ is locally Lipschitz,
 for any initial condition $(x_{10}, x_{20}) \in
 X$ of (\ref{eq3}) the term $v(t)$ can be bounded as
 \[
 \ba{rcl}
 |v(t)| &=& |q(x_1(t),x_2(t)) - \bar q(x_1^\star(t)) + \bar q(x_1^\star(t))- \bar q(\bar
 x_1(t))| \\
 &\leq& |q(x_1(t),x_2(t)) - \bar q(x_1^\star(t))| + |\bar q(x_1^\star(t))- \bar q(\bar
 x_1(t))| \\
 & \leq& L_q |(x_1(t),x_2(t))|_{{\rm graph} \left . \tau \right
 |_\ca} + 2 \sup_{s \in  {\ca}} |\bar q(s)|\\
  &\leq& v_M
 \ea
 \]
 for all $t\geq 0$, where $L_q$ is a bound of the Lipschitz constant of $q$ on the forward flow of (\ref{x1x2sys})
 originated from $X$ and $v_M$ a positive constant, both dependent on $X$.
 Hence, from estimate (\ref{estim1}), it follows that $|(\bar x_1(t), \xi(t))|_{{\rm graph} \left . \bar T
 \right |_{{\ca}}}$ is bounded and, since $\bar x_1(t) \in
 \ca$ and ${{\rm graph} \left . \bar T
 \right |_{{\ca}}}$ is compact, also $\xi(t)$ is bounded.
 This, in turn, implies that also the trajectories of (\ref{eq2})
 originated from $X \times \Real^{\bar \dimxi}$ are ultimately bounded,
 namely there exists a compact set $S \subset \Real^{\dimx_1+\dimx_2+\bar \dimxi} $
 such that for any $\Xi \subset \Real^{\bar \dimxi}$
 there exists a $T>0$ such that  $\forall \, (x_{10} , x_{20} , \xi_0) \in X \times \Xi$,
 $(x_1(t),x_2(t),\xi(t)) \in S$ for all $t \geq T$.
 As a consequence, the trajectories of  (\ref{eq2}) are uniformly attracted by $\omega(S)$,
 the $\omega$-limit set of  the set $S$ of system (\ref{eq2}), which is a bounded invariant set.

We prove now that $\omega(S)$ is a subset of ${\rm graph} \left .
T \right |_{{\cal R}}$. For, let $ (x_{10}', x_{20} ', \xi_0')$ be
a point of  $\omega(S)$ and note that, by (\ref{ineq11}) which
implies that ${\rm graph} \left . \tau \right |_\ca$ is uniform
attractive for the $(x_1,x_2) $ dynamics in (\ref{eq2}), and by
Proposition \ref{toolPropomega} in Appendix \ref{AppendixAuxRes},
it turns out that necessarily $x_{10}' \in \ca$ and $ x_{20}'
=\tau(x_{10}')$. Furthermore it can be proved that $\xi_0' =\bar
T(x_{10}')$. In fact, suppose that it is not true, namely that
there exists an $\epsilon>0$ such that
 \beeq{ \label{epsest}
|(x_{10} ', \xi_0')|_{{\rm graph} \left . \bar T \right |_\ca}
\geq \epsilon\,.
 }
 As $\omega(S)$ is invariant for (\ref{eq2}), for
any $(x_{10} , x_{20}, \xi_0) \in \omega(S)$ the corresponding
trajectory $(x_1(t), x_2(t), \xi(t))\in \omega(S)$ for all $t \in
\Real$, and thus $x_1(t) \in \ca \subset \bar X_1$ and $x_2(t) =
\tau(x_1(t))$ $\forall \, t \in \Real$. So, inequality
(\ref{estim1}) with $v=0$ yields that, using compactness of
$\omega(S)$, there exists a $t_\epsilon
>0$ such that for all $(x_{10} , x_{20}, \xi_0) \in \omega(S)$,
$|(x_1(t),\xi(t))|_{{\rm graph} \left . \bar T \right |_\ca}  <
\epsilon/2$ for all $t \geq t_{\epsilon}$. The previous facts,
specialized with $(x_{10},x_{20},\xi_0) =
\Phi_{(\ref{eq2})}(-t_\epsilon, (x_{10}', x_{20}', \xi_0'))$,
yield that $(x_{10}', x_{20}',
\xi_0')=\Phi_{(\ref{eq2})}(t_\epsilon, (x_{10}, x_{20}, \xi_0))$
are such that $|(x_{10}',\xi_0')|_{{\rm graph} \left . \bar T
\right |_\ca}  < \epsilon/2$ which contradicts (\ref{epsest}).
Hence, $\omega(S)$ is necessarily a subset of ${\rm graph} \left .
\bar T \right |_{ {\ca}}$. Since $S$ can be taken, without loss of
generality, such that ${\rm graph} \left . \bar T \right |_{
{\ca}} \subset {\rm int} S$, the previous facts prove
(\ref{beta3}) with the only exception that the class $\cal KL$
 function $\beta_3(\cdot,\cdot)$ is not necessarily locally exponential (see
 \cite{BI03}).
  To prove local exponential stability we follow a Lyapunov
  approach.
First, note that, by defining $p_1 = {\rm col}(x_1, \xi)$, the
first and third dynamics of (\ref{eq2}) can be rewritten as
\[
   \dot p_1 = F_1(p_1) + G_1(p_1,x_2)(x_2 - \tau(x_1))
\]
where $F_1(p_1) ={\rm col}(\bar f_1(x_1), \bar \phi(\xi) \bar
\psi(\xi) \bar q_1(x_1))$, $G_1(p_1,x_2) = {\rm col}(r_1(x_1,x_2),
\bar \psi(\xi)r_2(x_1,x_2))$ in which $r_1(\cdot,\cdot)$ and
$r_2(\cdot,\cdot)$ are properly defined smooth functions. By
assumption and by standard converse Lyapunov results, there exist
a smooth function $V_1 : {\cal D}({\rm graph} \left . \bar T
\right |_{\ca}) \to \Real$ and positive numbers $c_1$, $\underline
a_1 \leq \bar a_1$, such that
 \beeq{\label{Lyap1}
 {\partial V_1(p_1) \over \partial p_1}F_1(p_1) \leq -c_1 V_1(p_1)
 \qquad \forall \, p_1 \in {\cal D}({\rm graph} \left . \bar T
\right |_{\ca})
 }
  and
 \beeq{ \label{boundLyap1} \underline a_1 |p_1|^2_{{\rm graph} \left . \bar T \right
|_{\ca}} \leq V_1(p_1) \leq \bar  a_1 |p_1|^2_{{\rm graph} \left .
\bar T \right |_{\ca}} \qquad \forall\; p_1 \;\;  :
\;\;(x_1,x_2,\xi) \in S\,.
 }
  Similarly, by letting $p_2= {\rm col}
(x_1,x_2)$ and by rewriting (\ref{x1x2sys}) as $\dot p_2 =
F_2(p_2)$, it turns out that there exist a smooth function $V_2 :
{\cal D}({\rm graph} \left . \tau\right |_{\ca}) \to \Real$ and
positive numbers $c_2$, $\underline a_2 \leq \bar a_2$, such that
\beeq{ \label{Lyap2}
 {\partial V_2(p_2) \over \partial p_2}F_2(p_1) \leq -c_2 V_2(p_2)
 \qquad \forall \, p_2 \in {\cal D}({\rm graph} \left . \tau
\right |_{\ca}) }
 and
 \beeq{\label{boundLyap2} \underline a_2 |p_2|^2_{{\rm graph} \left . \tau \right |_{\ca}}
\leq V_2(p_2) \leq \bar  a_2 |p_2|^2_{{\rm graph} \left . \tau
\right |_{\ca}} \qquad \forall\; p_2 \;\;  : \;\;(x_1,x_2,\xi) \in
S\,.
 }
 Furthermore, note that there exists a positive $\bar \tau$ such
 that
\beeq{ \label{boundtau}
 |x_2 - \tau(x_1)| \leq \bar \tau |(x_1,x_2)|_{{\rm graph}
\left . \tau \right |_{\ca}} \qquad \forall \;  (x_1,x_2) \;\;  :
\;\;(x_1,x_2,\xi) \in S \,.
 }
 As a matter of fact, given $(x_1,x_2, \xi) \in S$, let $\bar x_1 \in \ca$ be such that $|x_1 -\bar x_1, x_2 -
 \tau(\bar x_1)| = |(x_1,x_2)|_{{\rm graph} \left .
\tau \right |_{\ca}}$. As $|x_1 - \bar x_1| \leq |(x_1,x_2)|_{{\rm
graph} \left . \tau \right |_{\ca}}$ and $|x_2 - \tau(\bar x_1)|
\leq |(x_1,x_2)|_{{\rm graph} \left . \tau \right |_{\ca}}$, and
denoting by  $\bar \tau'$ an upper bound of the Lipschitz constant
of $\tau$ on $S$, it turns out that
\[
\ba{rcl}
 |x_2 - \tau(x_1)| &=& |x_2 - \tau(\bar x_1) + \tau(\bar x_1)-
 \tau(x_1)|\\
 &\leq& |x_2 - \tau(\bar x_1)| + |\tau(\bar x_1)- \tau(x_1)|\\
 &\leq&
 |(x_1,x_2)|_{{\rm graph} \left . \tau \right |_{\ca}} + \bar
 \tau' |x_1 - \bar x_1|\\
 &\leq& (1 + \bar \tau') |(x_1,x_2)|_{{\rm graph} \left . \tau \right
 |_{\ca}}
 := \bar \tau |(x_1,x_2)|_{{\rm graph} \left .
\tau \right |_{\ca}} \ea
\]
 for all $(x_1,x_2,\xi) \in S$, namely (\ref{boundtau}) holds.
Consider now the candidate Lyapunov
 function $V(x_1,x_2,\xi) = V_1(p_1) + \beta V_2(p_2)$ for system
 (\ref{eq2}) with $\beta>0$. By (\ref{boundLyap1}) and
 (\ref{boundLyap2}), there exist positive numbers $\underline a \leq \bar
 a$ (dependent on $\beta$) such that
 \[
            \underline  a  |(x_1,\xi)|^2_{{\rm graph} \left .
\bar T \right |_{\ca}} +  \underline  a  |(x_1,x_2)|^2_{{\rm
graph} \left . \tau \right |_{\ca}}    \leq      V(x_1,x_2,\xi)
\leq   \bar  a |(x_1,\xi)|^2_{{\rm graph} \left . \bar T \right
|_{\ca}} +  \bar  a |(x_1,x_2)|^2_{{\rm graph} \left . \tau \right
|_{\ca}}
 \]

 By (\ref{Lyap1}), (\ref{Lyap2}), (\ref{boundLyap2}) and
 (\ref{boundtau}), and by the fact that $G_1(p_1)$ is locally Lipschitz,
 it turns out that there exists a $\beta^\star>0$ such that for
 all $\beta \geq \beta^\star$ and for all $(x_1,x_2,\xi) \in S$
 \[
  \left . \dot V(x_1,x_2,\xi) \right |_{(\ref{eq2})} \leq  - c \, V(x_1,x_2,\xi)
 \]
  where $c$ is a positive constant. Combining the previous facts with by
  (\ref{ineq11}), standard arguments yields (\ref{beta3}) with $\beta_3(\cdot,
  \cdot)$ a locally exponential class-{\cal KL} function. This, in
  turn, proves the proposition with the functions
  $(\varphi,\psi,\gamma)$ associated to the triplet $({\rm col}(f_1(x_1,x_2),\, f_2(x_1,x_2)),$ $q(x_1,x_2),
  {\rm graph} \left . \tau \right |_\ca)$ in the definition of LER
  given by $(\bar \varphi, \bar \psi, \bar \gamma)$. $\triangleleft$

 \end{proof}

\end{appendix}

\end{document}